\documentclass[leqno]{siamart1116}

\usepackage{graphics,graphicx,epstopdf}

\usepackage{longtable}
\usepackage[tight]{subfigure}
\usepackage{float}
\usepackage{makecell}
\usepackage{comment}
\usepackage{mathtools}
\usepackage{listings} 
\usepackage{tikz}
\usepackage[leqno]{amsmath}
\usepackage{amsxtra, amsfonts,amscd,amssymb,graphicx,epsf}

\usepackage[algo2e, vlined,ruled,linesnumbered]{algorithm2e}

\newcommand{\obar}[1]{\overline{#1}}
\newcommand{\ubar}[1]{\underline{#1}}

\numberwithin{equation}{section}

\definecolor{am}{RGB}{0,101,189}

\newtheorem{assumption}[theorem]{Assumption}

\definecolor{mygreen}{rgb}{0,.5,0}

\newcommand{\be}{\begin{equation}}
\newcommand{\ee}{\end{equation}}
\newcommand{\bee}{\begin{equation*}}
\newcommand{\eee}{\end{equation*}}
\newcommand{\bea}{\begin{eqnarray}}
\newcommand{\eea}{\end{eqnarray}}
\newcommand{\beaa}{\begin{eqnarray*}}
\newcommand{\eeaa}{\end{eqnarray*}}
\newcommand{\R}{\mathbb{R}}
\newcommand{\C}{\mathbb{C}}

\newcommand{\tr}{\mathrm{tr}}

\newcommand{\iprod}[2]{ \left\langle #1, #2 \right\rangle}
\newcommand{\half}{\frac{1}{2}}

\newcommand{\Tcal}{\mathcal{T}}

\newcommand{\sym}{{\mathrm{sym}}}

\newcommand{\diag}{\mathrm{diag}}

\newcommand{\M}{\mathcal{M}}

\newcommand{\Mcal}{\mathcal{M}}
\newcommand{\N}{{\mathbb{N}}}
\newcommand{\conv}{\mathrm{conv}}
\newcommand{\kg}{\kappa_g}
\newcommand{\kH}{\kappa_H}

\newcommand{\E}{\mathbf{E}}
\newcommand{\var}{\mathrm{Var}}

\newcommand{\BN}{\mathrm{BN}}
\newcommand{\bP}{\mathbf{P}_{\TM}}
\newcommand{\bPk}{\mathbf{P}_{T_{x_k} \Mcal}}
\newcommand{\St}{\mathrm{St}}
\newcommand{\Fr}{\mathrm{Fr}}
\newcommand{\st}{\mathrm{s.t.}}

\newcommand{\Diag}{\mathrm{Diag}}

\newcommand{\TM}{T_x \mathcal{M}}
\newcommand{\grad}{\mathrm{grad}\;\!}
\newcommand{\Hess}{\mathrm{Hess}\;\!}
\newcommand{\hess}{\mathrm{Hess}\;\!}

\newcommand{\Hc}{{\mathcal{H}^{\mathrm{c}}}}
\newcommand{\He}{{\mathcal{H}^{\mathrm{e}}}}

\newcommand{\hf}{{\mathrm{hf}}}
\newcommand{\ks}{{\mathrm{ks}}}
\newcommand{\f}{{\mathrm{f}}}
\newcommand{\ion}{{\mathrm{ion}}}
\newcommand{\xc}{{\mathrm{xc}}}

\newcommand{\Bcal}{{\mathcal{B}}}

\newcommand{\Ncal}{\mathcal{N}}
\newcommand{\rk}{\mbox{ rank }}
\newcommand{\ddt}{\frac{D}{d t} \frac{d}{d t}}
\newcommand{\Sbb}{{\mathbb{S}}}

\newcommand{\Vcal}{{\mathcal{V}}}

\newcommand{\argmin}{\mathop{\mathrm{arg\, min}}}

\begin{document}

\tableofcontents
\newpage 
	
\title{A brief introduction to manifold optimization}
\author{Jiang Hu\thanks{Beijing International Center for Mathematical
		Research, Peking University, China (\email{jianghu@pku.edu.cn})} 
	    \and Xin Liu \thanks{State Key Laboratory of Scientific and Engineering Computing, Academy of Mathematics and Systems Science, Chinese Academy of Sciences, and University of Chinese Academy of Sciences, China (email: liuxin@lsec.cc.ac.cn). Research supported in part by NSFC grants 11622112 and 11688101, the National Center for Mathematics and Interdisciplinary Sciences, CAS, and Key Research Program of Frontier Sciences QYZDJ-SSW-SYS010, CAS.}
        \and Zaiwen Wen\thanks{Beijing International Center for Mathematical
        Research, Peking University, China (\email{wenzw@pku.edu.cn}). Research supported in part by the NSFC grants 11421101 and 11831002, and by the National
		Basic Research Project under the grant 2015CB856002.}
	    \and Yaxiang Yuan\thanks{State Key Laboratory of Scientific and Engineering Computing, Academy of Mathematics and
	    Systems Science, Chinese Academy of Sciences, Beijing, China
    (\email{yyx@lsec.cc.ac.cn}). Research supported in part by NSFC grants 11331012 and 11461161005.}
       }
\maketitle

\begin{abstract}
     Manifold optimization is ubiquitous in computational and applied mathematics, statistics, engineering, machine learning, physics, chemistry and etc.  One of the main challenges usually is the non-convexity of the manifold constraints. By utilizing the geometry of manifold, a large class of constrained optimization problems can be viewed as unconstrained optimization problems on manifold. From this perspective, intrinsic structures, optimality conditions and numerical algorithms for manifold optimization are investigated.  Some recent progress on the theoretical results of manifold optimization are also presented.     
\end{abstract}

\section{Introduction}
Manifold optimization is concerned with the following optimization problem
\be  \label{prob} \begin{aligned}
	\min_{x \in \M}  \quad & f(x), \\
\end{aligned}
  \ee
where $\M$ is a Riemannian manifold and $f$ is a real-valued function on $\M$,
which can be non-smooth. If additional constraints other than the manifold
constraint are involved, we can add in $f$ an indicator function of the feasible
set of these additional constraints. Hence, \eqref{prob} covers a general formulation for manifold optimization.
In fact, manifold optimization has been widely used in computational and applied mathematics, statistics, machine learning, data science, material science and so on. The existence of the manifold constraint is one of the main difficulties in algorithmic design and theoretical analysis. 

\textbf{Notations.} Let $\R$ and $\C$ be the set of real and complex numbers. For a matrix $X \in \C^{n \times p}$, $\bar{X}, X^*, \Re X$ and $\Im X$ are its complex conjugate, complex conjugate transpose, real and imaginary parts, 
respectively. Let $\Sbb^n$ be the set of all $n$-by-$n$ real symmetric matrices.
For a matrix $M \in \C^{n\times n}$, $\diag(M)$ is a vector in $\C^{n}$
formulated by the diagonal elements of $M$. For a vector $c \in \C^n$,  $\Diag(c)$ is an $n$-by-$n$ diagonal matrix with the elements of $c$ on the diagonal. 
For a differentiable function $f$ on $\M$, let $\grad f(x)$ and $\hess f(x)$ be
its Riemannian gradient and Hessian at $x$, respectively. If $f$ can be extended
to the ambient Euclidean space, we denote its Euclidean gradient and Hessian by
$\nabla f(x)$ and $\nabla^2 f(x)$, respectively. 

This paper is organized as follows. In \cref{sec:application}, various kinds of
applications of manifold optimization are presented. We review geometry on
manifolds, optimality conditions as well as state-of-the-art algorithms for
manifold optimization in \cref{sec:algorithm}. For some selected practical
applications in \cref{sec:application}, a few  theoretical results based on
manifold optimization are introduced in \cref{sec:analysis}. 

\section{Applications of manifold optimization} \label{sec:application}
In this section, we introduce applications of manifold optimization in $p$-harmonic flow, max-cut problems, phase retrieval, eigenvalue problem, electronic structure calculations, Bose-Einstein condensates, cryo-electron microscopy (Cryo-EM), combinatorial optimization, deep learning and etc.

\subsection{$P$-harmonic flow} 
$P$-harmonic flow is used in the color image recovery and medical image analysis. For instance, in medical image analysis, the human brain is often mapped to a unit sphere via a conformal mapping, see \cref{fig:brain}.
\begin{figure}[ht] 
	\centering
	\includegraphics[width=0.7\textwidth,height=0.3\textwidth]{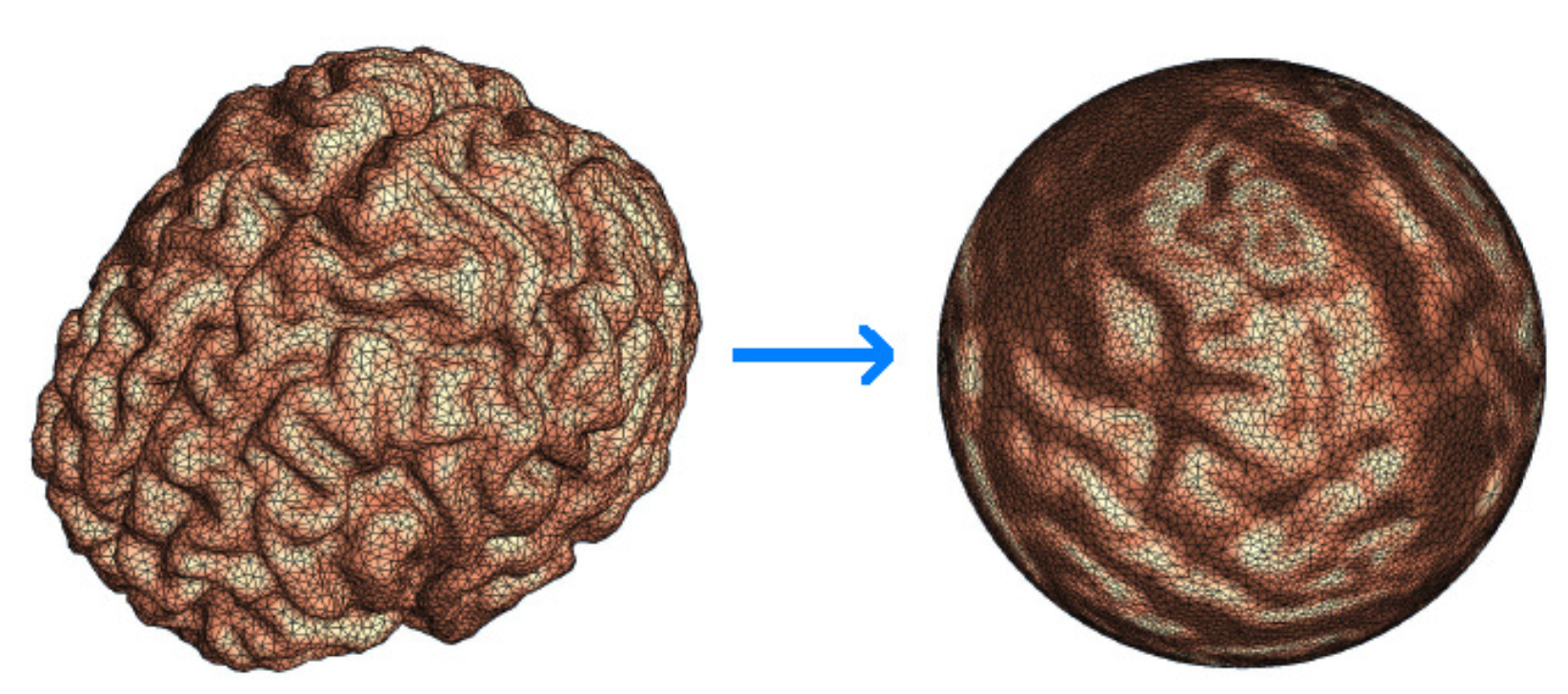} 
	\caption{conformal mapping between the human brain and the unit sphere \cite{lai2014folding}.}
	\label{fig:brain}
\end{figure}
By establishing a conformal mapping between an irregular surface and the unit
sphere, we can handle the complicated surface with the simple parameterizations
of the unit sphere. Here, we focus on the conformal mapping between genus-0
surfaces. From \cite{schoen1997lectures}, a diffeomorphic map between two
genus-0 surfaces $\Ncal_1$ and  $\Ncal_2$ is conformal if and only if it is a
local minimizer of the corresponding harmonic energy. Hence, one effective way
to compute the conformal mapping between two genus-0 surfaces is to minimize the
harmonic energy of the map. Before introducing the harmonic energy minimization
model and the diffeomorphic mapping, we review some related concepts on
manifold. Let $\phi_{\Ncal_1}(x^1,x^2):\R^2 \rightarrow \Ncal_1 \subset \R^3, \;
\phi_{\Ncal_2}(x^1, x^2):\R^2 \rightarrow \Ncal_2 \subset \R^3$ be the local
coordinates on $\Ncal_1$ and $\Ncal_2$, respectively. The first fundamental form
on $\Ncal_1$ is $g = \sum_{ij} g_{ij} dx^idx^j$, where $g_{ij} = \frac{\partial
\phi_{\Ncal_1}}{\partial x^i} \cdot \frac{\partial \phi_{\Ncal_1}}{\partial
x^j}, ~i,j=1,2 $. The first fundamental form on $\Ncal_2$ is $h = \sum_{ij}
h_{ij} dx^idx^j$, where $h_{ij} = \frac{\partial \phi_{\Ncal_2}}{\partial x^i}
\cdot \frac{\partial \phi_{\Ncal_2}}{\partial x^j}, ~i,j=1,2 $. Given a smooth map $f~:~\Ncal_1 \rightarrow \Ncal_2$, whose local coordinate representation is $f(x^1, x^2) = (f_1(x^1, x^2), f_2(x^1, x^2))$, the density of the harmonic energy of $f$ is 
\bee  e(f) = \|\mathrm{d}f\|^2 = \sum_{i,j=1,2}g^{ij}\langle f_*\partial_{x^i}, f_*\partial_{x^j}\rangle_{h}, \eee
where $(g^{ij})$ is the inverse of $(g_{ij})$ and the inner product between $f_*\partial_{x^i}$ and $f_*\partial_{x^j}$ is defined as:
\bee
\begin{split}
	\ \iprod{f_*\partial_{x^i}}{f_*\partial_{x^j}}_{h}=
	\iprod {\sum_{m=1}^2\frac{\partial f_m}{\partial x^i}\partial_{y_m}} {\sum_{n=1}^2\frac{\partial f_n}{\partial x^j}\partial_{y_n}}_{h}
	=  \sum_{m,n=1}^2h_{mn}\frac{\partial
		f_m}{\partial x^i}\frac{\partial f_n}{\partial x^j}.
\end{split}
\eee
This also defines a new Riemannian metric on $\Ncal_1$, $f^*(h)(\vec{v_1},\vec{v_2}) :=\langle f_*(\vec{v_1}), f_*(\vec{v_2})\rangle_{h}$, which is called the pullback metric induced by $f$ and $h$. Denote by $\mathbb{S}(\Ncal_1,\Ncal_2)$ the set of smooth maps between $\Ncal_1$ and $\Ncal_2$. Then the harmonic flow minimization problem solves
\bee
\min_{f\in\mathbb{S}(\Ncal_1,\Ncal_2)} 
\E(f) = \frac{1}{2}\int_{\Ncal_1}e(f)\mathrm{d}\Ncal_1,
\eee
 where $\E(f) $ is called the harmonic energy of $f$. Stationary points of $\E$
 are the harmonic maps from $\Ncal_1$ to $\Ncal_2$. In particular, If $\Ncal_2 =
 \mathbb{R}^2$, the conformal map $f = (f_1,f_2)$ is two harmonic functions
 defined on $\Ncal_1$. If we consider a $p$-harmonic map from  $n$ dimensional manifold $\M$ to $n$ dimensional sphere $S^n \subset \R^{n+1}$, the $p$-harmonic energy minimization problem can be written as
\bee
\label{eqn:pharmenergySn}
\begin{aligned}\min_{\vec{F}(x)=(f_1(x),\cdots,f_{n+1}(x)) } & \quad \E_p(\vec{F}) = \frac{1}{p}\int_{\M}\left(\sum_{k=1}^{n+1}\|\nabla_{\M}f_k\|^2\right)^{p/2}\mathrm{d}\M  \\
    \st   \quad\quad \quad\quad &  \vec{F}(x) \in S^n, ~\quad \forall x\in\M.
\end{aligned}
\eee

\subsection{Max cut}
Given a graph $G = (V,E)$ with a set of $n$ vertexes $V~(|V| = n)$ and a set of
edges $E$. Denote by the weight matrix $W=(w_{ij})$. The max-cut problem is to split $V$ into two nonempty sets $(S, V\backslash S)$ such that the total weights of edges in the cut is maximized. For each vertex $i=1,\ldots, n$, we define $x_i = 1$ if $i\in S$ and $-1$ otherwise. The maxcut problem can be written as
\be \label{prob:maxcut} \max_{x \in \R^n} \; \half \sum_{i<j} w_{ij}(1-x_i x_j), \; \st \; \; x_i^2 = 1, \;
i=1,\ldots, n.\ee
It is NP-hard. By relaxing the rank-1 constraint $xx^\top$ to a positive semidefinite matrix $X$ and further neglecting the rank-1 constraint on $X$, we obtain the following semidefinite program (SDP)
\be \label{prob:maxcut-sdp} \max_{ X \succeq 0} \quad   \tr(CX), \;
\st \; X_{ii} = 1, \quad i =1,\cdots,n,  \ee
where $C$ is the graph Laplacian matrix divided by $4$, i.e., $C=-\frac{1}{4}(\diag(We)-W )$. 
If we decompose $X=V^\top V$ with $V:=[V_1, \ldots,
V_n] \in \R^{p \times n}$, a nonconvex relaxation of \eqref{prob:maxcut} is
\be \label{prob:maxcut-noncvxrelax} \max_{V=[V_1, \ldots, V_n]} \; \tr(CV^\top V), \; \st \;
\|V_i\|_2 = 1, \;
i=1,\ldots, n.\ee
It is an optimization problem over multiple spheres.

\subsection{Low-rank nearest correlation estimation}

Given a symmetric matrix $C \in \Sbb^n$ and a non-negative symmetric weight
matrix $H \in \Sbb^n$, this problem is to find a correlation matrix $X$ of low
rank such that the distance weighted by $H$ between $X$ and $C$ is minimized:  
\be \label{prob:rNCM} \min_{ X \succeq 0} \; \half \| H \odot (X - C) \|_F^2, \;\; \st \;
X_{ii} = 1, \; i = 1, \ldots, n, \; \rk(X) \le p.\ee
Algorithms for solving \eqref{prob:rNCM} can be found in
\cite{SimonAbell2010,GaoSun2010}.  Similar to the maxcut problem, we decompose the low-rank matrix $X$ with $X = V^\top V$, in which $ V = [V_1, \ldots, V_n]  \in \R^{p \times n}$. Therefore,
problem \eqref{prob:rNCM} is converted to a quartic polynomial optimization
problem over multiple spheres:
\bee \label{prob:rNCM-V} \min_{ V \in \R^{p \times n}} \; \half \| H \odot
(V^\top V - C) \|_F^2, \;
\st \; \|V_i\|_2 = 1, \; i = 1, \ldots, n. \eee

\subsection{Phase retrieval}
Given some modules of a complex signal $x\in \mathbb C^n$ under linear measurements, a classic model for phase retrieval is to solve
\begin{equation}
\label{eq:Orignal}
\begin{aligned}
\mathrm{find} & \quad x \in \mathbb{C}^n \\
\hbox{s.t.} & \quad  |Ax|=b,
\end{aligned}
\end{equation}
where $A\in \mathbb{C}^{m\times n} $ and $b\in \mathbb{R}^m$. This problem plays
an important role in X-ray, crystallography imaging, diffraction imaging and
microscopy. Problem \eqref{eq:Orignal} is equivalent to the following problem,
which minimizes the phase variable $y$ and signal variable $x$ simultaneously:
\bee
\label{eq:deformation}
\begin{aligned} \min_{x\in\mathbb{C}^n, y\in \mathbb{C}^m} &\quad  \|Ax-y\|_2^2 \\
	\hbox{ s.t.} \quad &\quad  |y|=b.
\end{aligned}
\eee
In \cite{waldspurger2015phase}, the problem above is rewritten as 
\begin{equation}
\label{eq:phasecut1} \begin{aligned} \min_{x\in \mathbb{C}^n,u\in \mathbb{C}^m}
& \half \| Ax - \diag\{b\}u \|_2^2 \\
\hbox{s.t.} \quad &  |u_i|=1,i=1,\dots,m.
\end{aligned}
\end{equation}
For a fixed phase $u$, the signal $x$ can be represented by $x=A^{\dag}\diag\{b\}u$.
Hence, problem \eqref{eq:phasecut1} is converted to
\begin{equation}
\label{eq:phasecut2}
\begin{aligned}
\min_{u \in \mathbb{C}^m}   \quad& u^\ast M u\\
\hbox{s.t.}  \quad & |u_i|=1,i=1,\dots,m,
\end{aligned}
\end{equation}
where $M=\diag\{b\}(I-AA^{\dag})\diag\{b\}$ is positive definite. 
It can be regarded as a generalization of the maxcut problem to complex spheres. 

If we denote $X = uu^*$, \eqref{eq:phasecut2} can also be modelled as the following SDP problem \cite{cai2019fast}
\[ \min \quad \tr(MX) \quad \st \;\; X \succeq 0, \; \rk(X) = 1, \]
which can be further relaxed as 
\[ \min \quad \tr(MX)\quad \st \;\; \rk(X) = 1, \] 
whose constraint is a manifold. 

\subsection{Bose-Einstein condensates}
In Bose-Einstein condensates (BEC), the total energy functional is defined as
\[ E(\psi) = \int_{\R^d} \left[ \half |\nabla \psi(w)|^2 + V(w)|\psi(w)|^2 + \frac{\beta}{2}|\psi(w)|^4 - \Omega \bar{\psi}(w)L_z(w)\right] dw,  \]
where $w\in \R^d$ is the spatial coordinate vector, $\bar{\psi}$ is the complex
conjugate of $\psi$, $L_z = -i(x\partial - y\partial x),\, V(w)$ is an external
trapping potential,  and $\beta, \Omega$ are given constants. The ground state of BEC is defined as the minimizer of the following optimization problem
\[ \min_{\phi \in S } \quad E(\phi), \]
where the spherical constraint $S$ is
\[ S = \left \{ \phi~:~E(\phi) \leq \infty, \; \int_{\R^d} |\phi(w)|^2 dw= 1 \right \}.  \]
The Euler-Lagrange equation of this problem is to find $(\mu \in \R, \, \phi(w))$ such that
\[ \mu \phi(w) = -\half \nabla^2 \phi(w) + V(w) \phi(w) + \beta |\phi(w)|^2 \phi(w) - \Omega L_z \phi(w), \; \xi \in \R^d, \] 
and
\[ \int_{\R^d} |\phi(w)|^2 dw = 1. \]
Utilizing some proper discretization, such as finite difference, sine pseudospectral and Fourier pseudospectral methods, we obtain a discretized BEC problem 
\[ \min_{x \in \mathbb{C}^M} ~ f(x) := \frac{1}{2} x^*Ax + {\frac{\beta}{2}}\sum_{j =1}^M |x_j|^4, \quad \st \quad \|x\|_2 = 1, \]
where $M \in \N$, $\beta$ are given constants and $A \in \mathbb{C}^{M\times M}$ is Hermitian. Consider the case that $x$ and $A$ are real. Since $x^\top x=1$, multiplying the quadratic term of the objective function by $x^\top x$, we obtain the following equivalent problem
\bee \label{eq:d1}
\,\,\left\{
\begin{array}{lll}
	\displaystyle{\min_{x \in \mathbb{R}^M}} & f(x) = \frac{1}{2}x^{*\top}Axx^{\top}x + \frac{\beta}{2}\sum_{i=1}^{M}|x_i|^4\\
	\mbox{s.t.} & \|x\|_2 = 1.
\end{array}
\right.
\eee
The problem above can be also regarded as the best rank-1 tensor approximation of a fourth-order tensor $\mathcal{F}$ \cite{hu2016note}, with 
\bee \label{BEC-tensor}
\mathcal{F}_{\pi(i,j,k,l)}= \left \{
\begin{aligned}  a_{kl}/4 ,     \quad & i= j= k \ne l ,\\
	a_{kl}/12 ,     \quad  & i= j, i \ne k, i \ne l, k\ne l ,\\
	(a_{ii}+a_{kk})/12,  \quad & i= j\ne k =l ,\\
	a_{ii}/2+\beta/4 , \quad & i= j= k= l ,\\
	0           ,  \quad & \mbox{otherwise}.
\end{aligned}
\right.
\eee
For the complex case, we can obtain a best rank-1 complex tensor approximation problem by a similar fashion.
Therefore, BEC is an polynomial optimization problem over single sphere.

\subsection{Cryo-EM}

The Cryo-EM problem is to reconstruct a three-dimensional object from a series
of two-dimensional projected images $\{P_i\}$ of the object.
A classic model formulates it into an optimization problem over multiple orthogonality constraints 
\cite{singer2011three} to compute the $N$ corresponding directions $\{\tilde{R}_i\}$ of $\{P_i\}$, see \cref{fig:cryo2}.
Each $\tilde{R}_i\in\mathbb{R}^{3\times 3}$ is a three-dimensional rotation, i.e., $\tilde{R}^\top_i\tilde{R}_i = I_3$ and $\det(\tilde{R}_i)=1$. 
\begin{figure}[htp] 
	\centering
	\includegraphics[width=0.7\textwidth,height=0.5\textwidth]{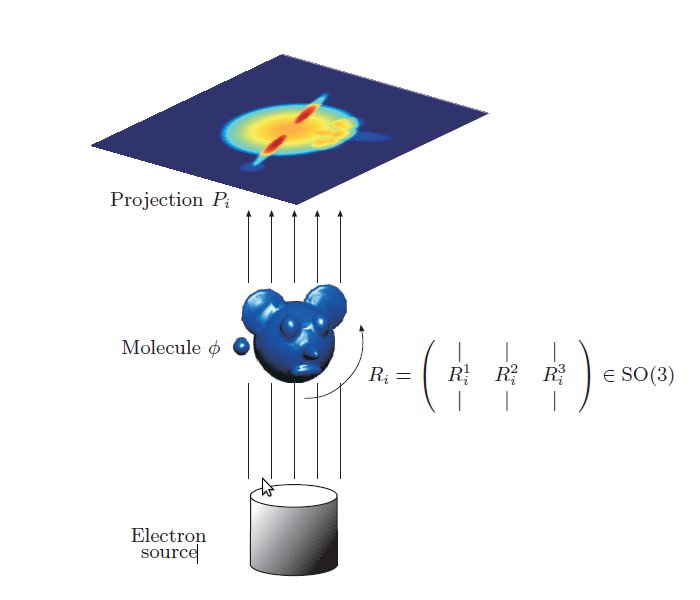} 
	\caption{Recover the 3-D structure from 2-D projections \cite{singer2011three}.}
	\label{fig:cryo2}
\end{figure}
Let $\tilde{c}_{ij} = (x_{ij},y_{ij},0)$ be the common line of $P_i$ and $P_j$ (viewed in $P_i$).
If the data are exact, it follows from the Fourier projection-slice theorem \cite{singer2011three}, the common lines coincide, i.e.,
\[\tilde{R}_i\tilde{c}_{ij}=\tilde{R}_j\tilde{c}_{ji}.\]
Since the third column of $\tilde{R}^3_i$ can be represented by the first two columns $\tilde{R}_i^1$ and $\tilde{R}_i^2$ as
$\tilde{R}^3_i=\pm\tilde{R}^1_i\times\tilde{R}^2_i$, the rotations $\{\tilde{R}_i\}$ can be compressed as a 3-by-2 matrix.  
Therefore, the corresponding optimization problem is
\begin{equation} \label{prob:cryo}
\begin{aligned}
\min_{R_i}  \quad \sum^N_{i=1}\rho(R_i c_{ij},R_j c_{ji}),\quad                     
\mathrm{s.t.}  \quad R^\top_iR_i=I_2,R_i\in\R^{3\times 2},
\end{aligned}
\end{equation}
where $ \rho$ is a function to measure the distance between two vectors,
$R_i$ are the first two columns of $\tilde{R}_i$ and $c_{ij}$ are the first two entries of $\tilde{c}_{ij}$. In \cite{singer2011three}, the distance function is set as $\rho(u,v)=\|u-v\|_2^2$. An eigenvector relaxation and SDP relaxation are also presented in \cite{singer2011three}.

\subsection{Linear eigenvalue problem}
Linear eigenvalue decomposition and singular value decomposition are the special cases of optimization with orthogonality constraints. Linear eigenvalue problem can be written as
\be \label{prob:eig} \min_{X \in \R^{n\times p}} \quad \tr(X^\top AX), \quad \st \; X^\top X = I, \ee
where $A \in \Sbb^{n}$ is given. Applications from low rank matrix optimization,
data mining, principal component analysis and high dimensionality reduction
techniques often need to deal with large-scale dense matrices or matrices with
some special structures. 
 Although modern computers are developing
 rapidly, most of the current eigenvalue and singular value decomposition
 softwares are limited by the traditional design and implementation. In
 particular, the efficiency may not be significantly improved when working with
 thousands of CPU cores. From the perspective of optimization,  a series of fast
 algorithms for solving \eqref{prob:eig} was proposed in \cite{liu2015efficient,
 liu2013limited, wen2016trace, wen2017accelerating}, whose essential parts can
 be divided into two steps, updating a subspace to  approximate the
 eigenvector space better and extracting eigenvectors by the Rayleigh-Ritz (RR)
 process. The main numerical algebraic technique for updating subspaces 
  is usually based on the Krylov subspace, which constructs a series of
 orthogonal bases sequentially. 
In \cite{wen2016trace}, the authors propose an equivalent unconstrained penalty function model
\bee
\min_{X \in \R^{n \times p}}\; 
f_\mu(X) := \frac{1}{2}\tr(X^\top AX) +
\frac{\mu}{4}\|X^\top X-I\|^2_F,
\eee
where $\mu$ is a parameter. By choosing an appropriate finite large $\mu$, the
authors established its equivalence with \eqref{prob:eig}. When $\mu$ is chosen
properly, the number of saddle points of this model is less than that of
\eqref{prob:eig}. More importantly, the
model allows one to design an algorithm that uses only matrix-matrix
multiplication. A Gauss-Newton algorithm for calculating low rank decomposition
is developed in \cite{liu2015efficient}. 
When the matrix to be decomposed is of low rank, this algorithm can be more
effective while its complexity is similar to the gradient method but with $Q$ linear
convergence. Because the bottleneck of many current iterative algorithms is the
RR procedure of the eigenvalue decomposition of smaller dense matrices,
the authors of \cite{wen2017accelerating} proposed a unified augmented subspace
algorithmic framework. Each step iteratively solves a linear eigenvalue problem: 
\[ Y = \argmin_{X \in \R^{n \times p}} ~\{ \tr(X^\top AX)~: X^\top X = I, \; X \in \mathcal{S} \}, 
\] 
where $\mathcal{S}:= \mathrm{span} \{ X, AX, A^2X, \ldots, A^k X \}$ with a
small $k$ (which can be far less than $p$). 
By combining with the polynomial acceleration technique and
deflation in classical eigenvalue calculations, it needs only  one RR
procedure theoretically to reach a high accuracy. 

When the problem dimension reaches the magnitude of $O(10^{42})$, the scale of data storage far exceeds the extent that traditional algorithms can handle. In \cite{zhang2017subspace}, the authors consider to use a low-rank tensor format to express data matrices and eigenvectors. 
Let $N= n_1n_2\ldots n_d$ with positive integer $n_1, \ldots, n_d$. A vector $u\in \R^N$ can be reshaped as a tensor
$\mathbf{u} \in\R^{n_1\times n_2\times\cdots\times n_d}$, whose entries $u_{i_1i_2\dots 
	i_d}$ are aligned in reverse lexicographical order, $1\leq
i_{\mu}\leq n_{\mu}, \mu=1,2,\dots,d$. A tensor $\mathbf{u}$ can be written as the TT format if its entries can be represented by
\bee
\label{TTdef}
u_{i_1i_2\dots i_d}=U_1(i_1)U_2(i_2)\cdots U_d(i_d),
\eee
where $U_{\mu}(i_{\mu})\in \R^{r_{\mu-1}\times
	r_{\mu}},i_{\mu}=1,2,\dots,n_{\mu}$ and fixed  dimensions $r_\mu, \; \mu=0,1,\ldots,d$ with $r_0 = r_d = 1$. In fact, the components  $r_\mu$, $\mu=1,\ldots,d-1$ are often equal to a value $r$ ($r$ is then called the TT-rank). 
Hence, a vector $u$ of dimension $\mathcal{O}(n^d)$ can be stored with $\mathcal{O}(dnr^2)$ entries if the corresponding tensor $\mathbf{u}$ has a TT format. A graphical representation of $\mathbf{u}$
can be seen in \cref{fig:TT}.
The eigenvalue problem can be solved based on the subspace algorithm. By
utilizing the alternating direction method with suitable truncations, the performance of the algorithm can be further improved. 
\begin{figure}[!h] 
	\centering
	\begin{tikzpicture}
	\draw (0.5,0) -- (0.7,0.3);    
	\draw [thick](0.5+0.08+0.04,0+0.12+0.06) -- (1.3+0.08+0.04,0+0.12+0.06);         
	\draw [thick](0.5+0.08+0.04,0+0.12+0.06-2.5) -- (1.3+0.08+0.04,0+0.12+0.06-2.5);   
	\node at (0.5+0.08+0.45,0+0.12+0.06-2.7) {\tiny $\mathbf{U_1(i_1)}$};
	\node at (0.5+0.08+0.45,0+0.12+0.06-2.4) {\tiny $\mathbf{r_1}$};

	\draw (0.7,0.3) -- (1.5,0.3);  
	\draw (0.5,0) -- (0.3,-0.3);
	\draw (0.3,-0.3) -- (1.1,-0.3);
	\draw (1.1,-0.3) -- (1.5,0.3);
	\node at (0.9,-0.05) {\small $\mathbf{U} _1$};
	\node at (0.3,-0.4) {\tiny $\mathbf{r_0}$};
	\node at (0.32,0) {\tiny $\mathbf{n_1}$};
	\node at (0.8,-0.4) {\tiny $\mathbf{r_1}$};
	\draw (1.3,0) -- (1.8,0);
	\node at (1.65,-2.35) {$\times$};
	\draw (1.8,-0.4) -- (1.8,0.4);
	\draw [thick] [fill=gray](1.88,0.52) to (3.38,0.52) to (3.38,-0.28) to (1.88,-0.28) to (1.88,0.52);
	
	\draw [thick] [fill=gray](1.88,0.52-2.5) to (3.38,0.52-2.5) to (3.38,-0.28-2.5) to (1.88,-0.28-2.5) to (1.88,0.52-2.5);
	\node at (2.65, -2.95) {\tiny$\mathbf{U_2(i_2)}$};
	\node at (2.65, -1.9) {\tiny$\mathbf{r_2}$};
	\node at (1.75, -2.15) {\tiny$\mathbf{r_1}$};
	\draw (1.8,0.4) -- (3.3,0.4);  
	\draw (1.8,-0.4) -- (3.3,-0.4);
	\draw (3.3,-0.4) -- (3.3,0.4);
	\draw (1.8,0.4) -- (2,0.7);
	\draw (3.3,0.4) -- (3.5,0.7);
	\draw (2,0.7) -- (3.5,0.7);
	\draw (3.3,-0.4) -- (3.5,-0.1);
	\draw (3.5,0.7) -- (3.5,-0.1);
	\draw [dashed] (1.8,-0.4) -- (2, -0.1);
	\draw [dashed] (2,-0.1) -- (2,0.7);
	\draw [dashed] (2,-0.1) -- (3.5,-0.1);
	\node at (2.65,0.05)  {\small $\mathbf{U} _{2}$};
	\node at (1.67,-0.18) {\tiny $\mathbf{r_{1}}$};
	\node at (2.6,-0.52) {\tiny $\mathbf{r_{2}}$};
	\node at (1.68,0.58) {\tiny $\mathbf{n_{2}}$};
	\draw (3.3+0.1,0 ) -- (4,0);
	\node at (3.75,-2.35) {$\times$};
	\draw [thick][fill = gray] (4+0.1,0.75+0.15) to (4+0.1+1.5,0.75+0.15) to (4+0.1+1.5,0.75+0.15-1.5) to (4+0.1,0.75+0.15-1.5) to (4+0.1,0.75+0.15);
	\draw [thick][fill = gray] (4+0.1,0.75+0.15-2.5) to (4+0.1+1.5,0.75+0.15-2.5) to (4+0.1+1.5,0.75+0.15-1.5-2.5) to (4+0.1,0.75+0.15-1.5-2.5) to (4+0.1,0.75+0.15-2.5);
	
	\node at (4.9,-3.28) {\tiny$\mathbf{U_3(i_3)}$};
	\node at (4.9,-1.52) {\tiny$\mathbf{r_3}$};
	\node at (3.95,-2.05) {\tiny$\mathbf{r_2}$};
	\draw (4.4-0.4,-0.55+1.3) -- (4.4-0.4,-2.05+1.3);
	\draw (4.4-0.4,-0.55+1.3) -- (5.9-0.4,-0.55+1.3);
	\draw (4.4-0.4,-2.05+1.3) -- (5.9-0.4,-2.05+1.3);
	\draw (5.9-0.4,-0.55+1.3) -- (5.9-0.4,-2.05+1.3);
	\draw (4.4-0.4,-0.55+1.3) -- (4.6-0.4,-0.25+1.3);
	\draw (5.9-0.4,-0.55+1.3) -- (6.1-0.4,-0.25+1.3);	
	\draw (4.6-0.4,-0.25+1.3) -- (6.1-0.4,-0.25+1.3);
	\draw (5.9-0.4,-2.05+1.3) -- (6.1-0.4,-1.75+1.3);
	\draw (6.1-0.4,-1.75+1.3) -- (6.1-0.4,-0.25+1.3);
	\draw [dashed] (4.4-0.4,-2.05+1.3) -- (4.6-0.4,-1.75+1.3);
	\draw [dashed] (4.6-0.4,-1.75+1.3) -- (4.6-0.4,-0.25+1.3);
	\draw [dashed] (4.6-0.4,-1.75+1.3) -- (6.1-0.4,-1.75+1.3);
	\node at (5.27-0.45,-1.28+1.3) {\small $\mathbf{U} _{3}$};
	\node at (3.85,-1.65+1.45) {\tiny $\mathbf{r_{2}}$};
	\node at (4.8,-2.22+1.3) {\tiny $\mathbf{r_{3}}$};
	\node at (4.33-0.4,-0.38+1.3) {\tiny $\mathbf{n_{3}}$};
	\node at (6.33,-0.05) {$\mathbf{\cdots\cdots}$};
	
	\draw [thick][fill = gray](6.9+0.12,0.75+0.18) to (6.9+0.12+0.9,0.75+0.18) to (6.9+0.12+0.9,0.75+0.18-1.5) to (6.9+0.12,0.75+0.18-1.5) to (6.9+0.12,0.75+0.18); 
	\draw [thick][fill = gray](6.9+0.12,0.75+0.18-2.5) to (6.9+0.12+0.9,0.75+0.18-2.5) to (6.9+0.12+0.9,0.75+0.18-1.5-2.5) to (6.9+0.12,0.75+0.18-1.5-2.5) to (6.9+0.12,0.75+0.18-2.5);
	\node at (5.8,-2.35) {$\times$};
	\node at (6.83,-2.35) {$\times$};
	\draw (6.9,0.75) -- (6.9,-0.75);
	\draw (6.9,0.75) -- (7.8,0.75);
	\draw (6.9,-0.75) -- (7.8,-0.75);  
	\draw (7.8,0.75) -- (7.8,-0.75);
	\draw (6.9,0.75) -- (7.1,1.05);
	\draw (7.1,1.05) -- (8,1.05);
	\draw (8,1.05) -- (7.8,0.75);
	\draw (7.8,-0.75) -- (8,-0.45);
	\draw (8,-0.45) -- (8,1.05);
	\draw [dashed](6.9,-0.75) -- (7.1,-0.45);
	\draw [dashed](7.1,-0.45) -- (7.1,1.05);
	\draw [dashed](7.1,-0.45) -- (8,-0.45);
	\node at (7.45,0) {\small $\mathbf{U_{d\textbf{-}1}}$};
	\node at (7.45,-3.25) {\tiny $\mathbf{U_{d\textbf{-}1}}(i_{d\textbf{-}1})$};
	\node at (7.45,-1.47) {\tiny $\mathbf{r_{d\textbf{-}1}}$};
	\node at (6.76,-2.0) {\tiny $\mathbf{r_{d\textbf{-}2}}$};
	\node at (8.87,-2.1) {\tiny $\mathbf{r_{d\textbf{-}1}}$};

	\node at (6.65,0.1) {\tiny $\mathbf{r_{d\textbf{-}2}}$};
	\node at (7.4,-0.88) {\tiny $\mathbf{r_{d\textbf{-}1}}$};
	\node at (6.75,0.93) {\tiny $\mathbf{n_{d\textbf{-}1}}$};
	\draw [thick] (8.5+0.08,0.45+0.12) to (8.5+0.08,0.45+0.12-0.9);
	\draw [thick] (8.5+0.08,0.45+0.12-2.5) to (8.5+0.08,0.45+0.12-0.9-2.5);
	\node at (6.33,-2.35){$\mathbf{\cdots\cdots}$};
	\node at (8.25,-2.35) {$\times$};
	\draw (7.9,0) -- (8.5,0);
	\draw (8.5,0.45) -- (8.5,-0.45);
	\draw (8.5,0.45) -- (8.7,0.75);
	\draw (8.5,-0.45) -- (8.7,-0.15);
	\draw (8.7,0.75) -- (8.7,-0.15);
	\node at (8.95,0.05) {\small $\mathbf{U} _d$};
	\node at (8.3,0.1) {\tiny $\mathbf{r_{d\textbf{-}1}}$};
	\node at (8.5,-0.5) {\tiny $\mathbf{r_{d}}$};
	\node at (8.45,0.65) {\tiny $\mathbf{n_{d}}$};
	\node at (9.05,-2.45) {\tiny $\mathbf{U_d(i_d)}$};
	\end{tikzpicture}
	\caption{Graphical representation of a TT tensor of order $d$ with cores $\mathbf{U} _{\mu},~\mu = 1,2,\ldots ,d$. The first row is $\mathbf{u}$, the second row are its entries $u_{i_1i_2\ldots i_d}$.
	} \label{fig:TT}
\end{figure}
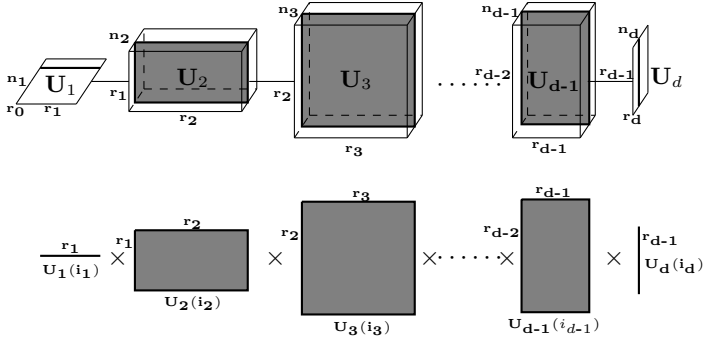

The online singular value/eigenvalue decomposition appears in principal
component analysis (PCA). The traditional PCA first reads the data and then
performs eigenvalue decompositions on the sample covariance matrices. If the
data is updated, the principal component vectors need be investigated again
based on the new data. Unlike traditional PCA, the online PCA reads the samples
one by one and updates the principal component vector in an iterative way, which
is essentially a random iterative algorithm  of the maximal trace optimization
problem. As the sample grows, the online PCA algorithm leads to more accurate
main components. An online PCA is proposed and analyzed in
\cite{oja1985stochastic}. It is proved that the convergence rate is $O(1/n)$
with high probability. A linear convergent VR-PCA algorithm is investigated in
\cite{shamir2015stochastic}. In \cite{li2018near}, the scheme in
\cite{oja1985stochastic} is  further proved that under the assumption of
Subgaussian's stochastic model, the convergence speed of the algorithm can reach
the minimal bound of the information, and the convergence speed is near-global.

\subsection{Nonlinear eigenvalue problem}
The nonlinear eigenvalue problems from electronic structure calculations are another important source of problems with orthogonality constraints,
such as Kohn-Sham (KS) and Hartree-Fock (HF) energy minimization problems. By properly discretizing, KS energy functional can be expressed as
\[ E_{\ks}(X) := \frac{1}{4}\tr(X^*LX) + \half \tr(X^*V_{\ion}X) + \half \sum_l \sum_i \zeta_l |x_i^* w_l|^2 + \frac{1}{4} \rho^\top L^\dag \rho + \half e^\top \epsilon_{\xc}(\rho), \]
where $X \in \mathbb{C}^{n \times p}$ satisfies $X^*X = I_p$, $\rho =
\diag(XX^*)$ is the charge density and $\mu_{\xc}(\rho) := \frac{\partial
\epsilon_{\xc}(\rho)}{\partial \rho}$ and $e$ are vectors in $\R^n$ with
elements all of ones. More specifically, $L$ is a finite dimensional representation of the Laplacian operator, $V_\ion$ is a constant example, $w_l$
represents a discrete reference projection function, $\zeta_l$ is a constant of $\pm 1$, $\epsilon_{\xc}$ is used to characterize exchange-correlation energy. With the KS energy functional, the KS energy minimization problem is defined as
\[ \min_{X \in \mathbb{C}^{n \times p}} \quad E_{\ks}(X) \quad \st \;\; X^*X = I_p. \]
Compared to the KS density functional theory, the HF theory can provide a more accurate model. Specifically, it introduces a Fock exchange operator, which is a fourth-order tensor by some discretization, $\Vcal(\cdot): \mathbb{ C}^{n \times n} \rightarrow \mathbb{C}^{n \times n}$. The corresponding Fock energy can be expressed as
\[ E_{\f} := \frac{1}{4} \iprod{\Vcal(XX^*)X}{X} = \frac{1}{4} \iprod{\Vcal(XX^* )}{XX^*}. \]
The HF energy minimization problem is then
\be \label{prob:hf} \min_{X \in \mathbb{C}^{n\times p}} \quad E_{\hf}(X) := E_{\ks}(X) + E_{\f}(X) \quad \st \;\; X^*X = I_p. \ee

The first-order optimality conditions of KS and HF energy minimization problems correspond to two different nonlinear eigenvalue problems. Taking KS energy minimization as an example, the first-order optimality condition is
\be \label{eq:ks} H_{\ks}(\rho) X = X \Lambda, \quad X^*X = I_p, \ee
where ${H_{\ks}(X)} := \half L + V_{\ion} + \sum_l \zeta_l w_lw_l^* + {\diag}((\Re L^\dag) \rho) + { \diag}(\mu_{\xc}(\rho)^* e)$ and $\Lambda$ are diagonal matrices. The equation \eqref{eq:ks} is also called the KS equation.
The nonlinear eigenvalue problem aims to find some orthogonal eigenvectors
satisfying \cref{eq:ks}, while the optimization problem with orthogonality
constraints minimizes the objective function under the same constraints. These
two problems are connected by the optimality condition and both describe the steady state of the physical system. 

The most widely used algorithm for solving the KS equation is the so-called self-consistent field iteration (SCF), which is to solve the following linear eigenvalue problems repeatedly
\be \label{alg:scf} H_{\ks}(\rho_k) X_{k+1} = X_{k+1} \Lambda_{k+1}, \quad X_{k+1}^*X_ {k+1} = I_p, \ee
where $\rho_k = \diag(X_k X_k^*)$.
In practice, to accelerate the convergence, we often replace the charge density
$\rho_k$ by a linear combination of the previously existing $m$ charge densities
\[ \rho_{mix} = \sum_{j=0}^{m-1} \alpha_j \rho_{k-j}. \]
In the above expression, $\alpha = (\alpha_0, \alpha_1, \ldots, \alpha_{m-1})$ is the solution to the following minimization problem:
\[ \min_{\alpha^\top e = 1} \quad \| R\alpha \|^2, \]
where $R= (\Delta \rho_k, \Delta \rho_{k-1}, \ldots, \Delta \rho_{k-m+1})$,
$\delta_j = \rho_j - \rho_{j-1}$ and $e $ is a $n$-dimensional vector of all
entries ones.
After obtaining $\rho_{mix}$, we replace $H_{\ks}(\rho_k)$ in \eqref{alg:scf} with
$H_{\ks}(\rho_{mix})$ and execute the iteration \eqref{alg:scf}.
This technique is called charge mixing. For more details, one can refer to \cite{pulay1980convergence,pulay1982improved,toth2017local}.

Since SCF may not converge, many
researchers have recently developed optimization algorithms for the electronic
structure calculation that can guarantee convergence. In
\cite{zhang2014gradient}, the manifold gradient method is directly extended to
solve the KS minimum problem. The algorithm complexity is mainly from the
calculation of the total energy and its gradient calculation, and the projection
on the Stiefel manifold. Its complexity at each step is much lower than the
linear eigenvalue problem and it is easy to be parallelized. Extensive numerical experiments based on the software packages Octopus and RealSPACES show that the algorithm is often more efficient than SCF.
In fact, the iteration \eqref{alg:scf} of SCF can be understood
as an approximate Newton algorithm in the sense that the
complicated part of the Hessian of the total energy is not considered:
\[ \min_{X \in \mathbb{C}^{n \times p}} \quad q(X) := \frac{1}{2} \tr(X^*H_{\ks}(\rho_k)X) \quad \st \;\; X^*X = I_p. \]
Since $q(X)$ is only a local approximation model of $E_{\ks}(X)$, there is no
guarantee that the above model ensures a sufficient decrease of $E_{\ks}(X)$. 

An explicit expression of the complicated part of the Hessian matrix is derived
in \cite{wen2013adaptive}. Although this part is not suitable for an explicit storage, its operation with a vector is simple and feasible.
Hence, the full Hessian matrix can be used to improve the reliability of
Newton's method. By adding regularization terms, the global convergence is also
guaranteed. 
A few other related works include \cite{dai2017conjugate, zhao2015riemannian, zhang2015maximization, gao2018new, lai2016localized}.

The ensemble-based density functional theory is especially important when the spectrum of the Hamiltonian matrix has no significant gaps. The KS energy minimization model is modified by allowing the charge density to contain more wave functions. Specifically,
$
\rho(r) = \sum_{i=1}^{p} f_i |\psi_i(r)|^2
$
where $p \geq p_e$ and the fraction occupation $0 \leq f_i \leq 1$ is to ensure that the total charge density of the total orbit is $p$, i.e.,
$
\sum_{i=1}^p f_i = p_e.
$
To calculate the fractional occupancy, the energy functional in the ensemble
model introduces a temperature $T$ associated with an entropy $\alpha R(f)$,
where $ \alpha := \kappa_B T$, $\kappa_B$ is the Boltzmann constant, $R(f)=\sum\limits_{i=1}^p s(f_i)$,
\bee \label{eq:def-s} s(t) = \begin{cases} t\ln t + (1-t)\ln(1-t), & 0 < t < 1,
	\\
	0, & \mbox{otherwise}.
\end{cases}
\eee
This method is often referred as the KS energy minimization model with temperature or the ensemble KS energy minimization model (EDFT).
Similar to the KS energy minimization model, by using the appropriate
discretization, the wavefunction can be represented with $X=[x_1, \ldots, x_p] \in \mathbb{C}^{n \times p}.$
The discretized charge density in EDFT can be written as
\bee \label{eq:rho}
\rho(X,f) := \diag(X \diag(f) X^*).
\eee
Obviously, $\rho(X,f)$ is real. The corresponding discretized energy functional is
\beaa \label{eq:KSDFT-eD} M(X,f)&=&\tr(\diag(f) X^* A X)
+ \frac{1}{2} \rho^\top
L^\dagger \rho + e^\top \epsilon_{\xc}(\rho) + \alpha R(f). \eeaa
The discretized EDFT model is
\be\label{prob:minKS} \begin{aligned} \min_{X \in \mathbb{C}^{n \times p}, f \in \R^p} & \quad
	M(X, f) \\
	\st \qquad &\quad X^* X = I_p, \\
	& \quad e^\top f = p_e \quad 0 \le f \le 1.
\end{aligned} \ee
Although SCF can be generalized to this model, its convergence is still not guaranteed.
 An equivalent simple model with only one-ball constraint is proposed  in
 \cite{ulbrich2015proximal}. It is solved by a proximal gradient method 
 such that the terms other than the entropy function term are linearized. An 
 explicit solution of the subproblem is then derived and the convergence of the
 algorithm is established. 

\subsection{Approximation models for integer programming}
Many optimization problems arising from data analysis are NP-hard integer programmings.
Spherical constraints and orthogonal constraints are often used to obtain
approximate solutions with high quality. Consider optimization problem over the
permutation matrices: 
\[ \min_{X \in \Pi_n} \quad f(X), \]
where $f(X): \R^{n \times n} \rightarrow \R^n$ is differentiable, and $\Pi_n$ is a collection of $n$-order permutation matrices
\[ \Pi_n := \{ X \in \R^{n \times n} ~:~ Xe = X^\top e = e, X_{ij} \in \{ 0,1 \} \}. \]
This constraint is equivalent to
\[ \Pi_n := \{ X \in \R^{n \times n} ~:~ X^\top X = I_n, X \ge 0 \}. \]
It is proved in \cite{jiang2016l_p} that it is equivalent to an
$L_p$-regularized optimization problem over the doubly stochastic matrices,
which is much simpler than the original problem. An estimation of the lower
bound of the non-zero elements at the stationary points are presented. Combining
with the cut plane method,  a novel gradient-type algorithm with negative
proximal terms is also proposed. 

Given $k$ communities $S_1, S_2, \ldots, S_k$ and the set of partition matrix $P_{n}^k$, where the partition matrix $X \in P_n^k$ means $X_{ij} = 1,~ i,j \in S_t, ~ t\in \{1, \ldots, k \}$, otherwise $X_{ij} = 0$.
Let $A$ be the adjacency matrix of the network, $d_i = \sum_j A_{ij}, i \in \{ 1, \ldots, n
\} $ and $\lambda = 1 / \|d\|_2$.
Define the matrix $C: = - (A - \lambda dd^\top)$. The community detection
problem in social networks is to find a partition matrix to maximize the
modularity function under the stochastic block model:
\be \label{prob:comm} \min_X \quad \iprod{C}{X} \quad \st \; X \in P_n^k. \ee
A SDP relaxation of \eqref{prob:comm} is
\beaa \label{prob:comm-sdp} \min_X & \quad & \iprod{C}{X} \\
\st & \quad & X_{ii} = 1, i = 1, \ldots, n, \\
& & 0 \leq X_{ij} \leq 1, \forall i,j, \\
& & X \succeq 0.
\eeaa
A sparse and low-rank completely positive relaxation technique is further
investigated in \cite{zhang2017sparse} to transform the model into an
optimization problem over multiple non-negative spheres:
\be \label{prob:comm-relax}
\begin{aligned}
	\min_{U \in \R^{n \times k}}\quad & \iprod{C}{UU^\top} \\
	\st \quad & \|u_{i}\|_2 = 1, i = 1, \ldots, n, \\
	& \|u_{i}\|_0 \leq p = 1, i = 1, \ldots, n, \\
	& U \geq 0,
\end{aligned}
\ee
where $1\leq p \leq r$ is usually taken as a small number so that $U$ can be
stored for large-scale data sets. The equivalence to the original problem is
proved theoretically and  an efficient row-by-row type block coordinate descent
method is proposed. 
In order to quickly solve  network problems whose dimension is more than 10
million, an asynchronous parallel algorithm is further developed.

\subsection{Deep learning}
Batch normalization is a very popular technique in deep neural networks. It
avoids internal covariance translation by normalizing the input of each neuron.
The space formed by its corresponding coefficient matrix can be regarded as a
Riemannian manifold. For a deep neural network, batch normalization usually
involves input processing before the nonlinear activation function. Define $x$
and $w$ as the outputs of the previous layer and the parameter vector for the
current neuron, the batch normalization of $z:= w^\top x$ can be written as
\[ \BN(z) = \frac{z - \E (z) }{\var(z)} = \frac{w^\top (x - \E(x))}{ \sqrt{w^ \top R_{xx} w } } = \frac{u^\top (x - \E(x))}{\sqrt{u^\top R_{xx} u }}, \]
where $u := w/ |w|$, $\E (z)$ is the expectation of random variable $z$ and $R_{xx}$ are the covariance matrices of $x$. From the definition, we have $\BN(w^\top x) = \BN(u^\top x )$ and
\[ \frac{\partial \BN(w^\top x) }{ \partial x} = \frac{\partial \BN(u^\top x)}{\partial x}, \quad \frac{\partial \BN(z) }{ \partial w} = \frac{1}{w} \frac{\partial \BN(z)}{\partial u}. \]
Therefore, the use of the batch standardization ensures that the model does not explode with large learning rates and that the gradient is invariant to linear scaling during propagation.

Since $\BN(c w^\top x) = \BN(w^\top x)$ holds for any constant $c$ , the optimization problem for deep neural networks using batch normalization can be written as
\[ \min_{X \in \Mcal} \quad L(X), \quad \Mcal = S^{n_1-1} \times \cdots S^{n_m-1} \times \ R^l, \]
where $L(X)$ is the loss function, $S^{n-1}$ is a sphere in $\R^n$ (can also be viewed as a Grassmann manifold), $n_1, \ldots , n_m$ are the dimensions of the weight vectors, $m$ is the number of weight vectors, and $l$ is the number of remaining parameters to be decided, including deviations and other weight parameters. For more information, we refer to \cite{cho2017riemannian}. 

\subsection{Sparse PCA} \label{subsec:nonsmooth-spca}
In the traditional PCA, the obtained  principle eigenvectors are usually not sparse, which leads to high computational cost for computing the principle components. Spare PCA \cite{jolliffe2003modified} wants to find principle eigenvectors with few non-zero elements. The mathematical formulation is
\be \label{prob:spca} \begin{aligned}
\min_{X \in \R^{n\times p}} \quad & -\tr(X^\top A^\top AX) + \rho \|X\|_1 \\
\st \quad & X^\top X = I_p,
\end{aligned} \ee
where $\|X\|_1 = \sum_{ij} |X_{ij}|$ and $\rho > 0$ is a trade-off parameter. When $\rho = 0$, this reduces to the traditional PCA problem. For $\rho >0$, the term $\|X\|_1$ plays a role to promote sparsity. Problem \eqref{prob:spca} is a non-smooth optimization problem on the Stiefel manifold. 

\subsection{Low-rank matrix completion} \label{subsec:nonsmooth-rlmc}
The low-rank matrix completion problem has important applications in computer vision, pattern recognitions, statistics, etc. It can be formulated as
\be \label{prob:lrmc} \begin{aligned}
	\min_X \quad & \rk(X) \\
	\st \quad & X_{ij} = A_{ij}, \; (i,j) \in \Omega,
\end{aligned} \ee
where $X$ is the matrix that we want to recover (some of its entries are known)
and $\Omega$ is the index set of observed entries. Due to the difficulty of the rank,
a popular approach is to relax it into a convex model using the nuclear norm.
The equivalence between this convex problem and the nonconvex problem
\eqref{prob:lrmc} is ensured under certain conditions.
Another way is to use a low rank decomposition on $X$ and then solve the
corresponding unconstrained optimization problem \cite{wen2012solving}. If the
rank of the ground-truth matrix $A$ is known, an alternative model for a fixed-rank matrix completion is  
\be \label{prob:fixrank-mc} \min_{X \in\R^{n\times p}} \;\; \|\mathbf{P}_{\Omega}(X - A)\|_F^2, \; \st \;\; \mathrm{rank}(X) = r,  \ee
 where $\mathbf{P}_{\Omega}$ is a projection with $\mathbf{P}_{\Omega}(X)_{ij} =
 X_{ij}, \; (i,j) \in \Omega$ and $0$ otherwise, and $r = \rk(A)$. The set
 $\Fr(m,n,r):= \{ X \in \R^{m \times n}~:~ \rk(X) = r \}$ is a matrix manifold,
 called fixed-rank manifold. The related geometry is analyzed in
 \cite{vandereycken2013low}. Consequently, problem \eqref{prob:fixrank-mc} can
 be solved by optimization algorithms on manifold. Problem
 \eqref{prob:fixrank-mc} can deal with Gaussian noise properly. For data sets
 with a few outliers,  the robust low-rank matrix completion problem (with the
 prior knowledge $r$) considers: 
\be \label{prob:r-fixrank-mc} \min_{X \in\R^{n\times p}} \;\; \|\mathbf{P}_{\Omega}(X - A)\|_1, \; \st \;\; \mathrm{rank}(X) = r,  \ee
where $\|X\|_1 =\sum_{i,j} |X_{ij}|$. Problem \eqref{prob:r-fixrank-mc} is a
non-smooth optimization problem on the fixed-rank matrix manifold. For some related algorithms for \eqref{prob:fixrank-mc} and \eqref{prob:r-fixrank-mc}, the readers can refer to \cite{wei2016guarantees,cambier2016robust}.

\subsection{Sparse blind deconvolution}\label{subsec:nonsmooth-scsbd}
Blind deconvolution is to recover a convolution kernel $a_0 \in \R^k$ and signal $x_0 \in \R^m$ from their convolution 
\[ y = a_0 \circledast x_0,  \]
where $y \in \R^m$. Since there are infinitely many pairs $(a_0, x_0)$
satisfying this condition, this problem is often ill-conditioned. To overcome
this issue, some regularization terms and extra constraints are necessary. The sphere-constrained sparse blind deconvolution reformulate the problem as
\[ \min_{a,x} \;\; \| y - a \circledast x\|_2^2 + \mu\|x\|_1, \; \; \st \; \; \|a\|_2 =1, \]
where $\mu$ is a parameter to control the sparsity of the signal $x$. This is a
non-smooth optimization problem on the product manifold of a sphere and $\R^m$.
Some related background and the corresponding algorithms can be found in \cite{zhang2017global}.

\subsection{Non-negative PCA}\label{subsec:cons-nnpca}
Since the principle eigenvectors obtained by the traditional PCA may not be
sparse, one can enforce the sparsity by adding non-negativity constraints. The problem is formulated as
\be \label{prob:npca} \min_{X \in \R^{n\times p}} \;\; \tr(X^\top A A^\top X) \; \; \st \;\; X^\top X = I_p,\; X \geq 0, \ee
where $A = [a_1, \ldots, a_k] \in \R^{n \times k}$ are given data points. Under
the constraints, the variable $X$ has at most one non-zero element in each row.
This actually helps to guarantee the sparsity of the principle eigenvectors. 
Problem \eqref{prob:npca} is an optimization problem with manifold and
non-negative constraints. Some related information can be found in \cite{zass2007nonnegative,montanari2016non}.
\subsection{$K$-means clustering}\label{subsec:cons-kmeans} 
 $K$-means clustering is a fundamental problem in data mining.
Given $n$ data points $(x_1, x_2, \ldots, x_n)$ where each data point is a
$d$-dimensional vector, $k$-means is to partition them into $k$ clusters $S:=\{
S_1, S_2, \ldots, S_k\}$ such that the within-cluster sum of squares is
minimized. Each data point belongs to the cluster with the nearest mean. The mathematical form is 
\be \label{prob:kmeans-1}  \min_S \; \sum_{i=1}^k\sum_{x \in S_i} \|x - c_i\|^2,  \ee
where $c_i = \frac{1}{\mathrm{card}(S_i)} \sum_{x \in S_i} x $ is the center of
$i$-th cluster and $\mathrm{card}(S_i)$ is the cardinality of $S_i$.
Equivalently, problem \eqref{prob:kmeans-1} can be written as
\cite{carson2017manifold,liu2017subspace,xie2018non}:
\be \label{prob:kmeans} \begin{aligned}
	\min_{Y \in \R^{n\times k}} \quad & \tr(Y^\top D Y) \\
	\st \quad & YY^\top \mathbf{1} = \mathbf{1}, \\
	\quad & Y^\top Y = I_k, \; Y\geq 0, 
\end{aligned}\ee
where $D_{ij}:=\|x_i - x_j\|^2$  is the squared Euclidean distance matrix.
Problem \eqref{prob:kmeans} is a minimization over the Stiefel manifold with linear constraints and non-negative constraints.

\section{Algorithms for manifold optimization} \label{sec:algorithm}
In this section, we introduce a few state-of-the-art algorithms for optimization
problems on Riemannian manifold. Let us start from the concepts of manifold optimization.

\subsection{Preliminaries on Riemannian manifold}
A $d$-dimensional manifold $\Mcal$ is a Hausdorff and second-countable topological
space, which is homeomorphic to the $d$-dimensional Euclidean space locally via a
family of charts. When the transition maps of intersecting charts are smooth,
the manifold $\Mcal$ is called a smooth manifold. Intuitively, the tangent space $T_x \Mcal $ at a point $x$ of a manifold $\Mcal$ is the set of the tangent vectors of all the curves at $x$. Mathematically, a tangent vector $\xi_x$ to $\M$ at $x$ is a mapping such that there exists a curve $\gamma$ on $\M$ with $\gamma(0) = x$, satisfying
\[ \xi_x u := \dot{\gamma}(0) u \triangleq \left. \frac{\mathrm{d}(u(\gamma(t)))}{\mathrm{d}t}\right|_{t=0}, \quad \forall~u \in \Im_x(\M), \]
where $\Im_x(\M)$ is the set of all
real-valued functions $f$ defined in a neighborhood of $x$ in $\M$.
Then, the tangent space $T_x\Mcal$ to $\Mcal$ is defined as the set of all
tangent vectors to $\M$ at $x$.  If $\Mcal$ is equipped with a smoothly varied
inner product $g(\cdot, \cdot):=\iprod{\cdot}{\cdot}_x$ on the tangent space,
then $(\Mcal,g)$ is a Riemannian manifold. In practice, different Riemannian
metrics may be investigated to design efficient algorithms. The Riemannian
gradient $\grad f(x)$ of a function $f$ at $x$ is an unique vector in $T_x \Mcal$ satisfying
\[ \iprod{\grad f(x)}{\xi}_x = Df(x)[\xi], \quad \forall \xi \in \TM, \]
where $Df(x)[\xi]$ is the derivative of $f(\gamma(t))$ at $t = 0$, $\gamma(t) $
is any curve on the manifold that satisfies $\gamma( 0) = x$ and
$\dot{\gamma}(0) = \xi$. The Riemannian Hessian $\hess f(x)$ is a mapping from
the tangent space $\TM$ to the tangent space $\TM$:
\[ \hess f(x)[\xi] := \tilde{\nabla}_\xi \grad f(x) \]
where $\tilde{\nabla}$ is the Riemannian connection \cite{opt-manifold-book}.
For a function $f$ defined on the manifold, if it can be extended to the ambient
Euclidean space $\R^{n \times p}$, we have its Riemannian gradient $\grad f$ and
Riemannian Hessian $\hess f$:
\be 
\begin{aligned}
	\grad f(X) & = \bP (\nabla f(X)), \\
	\hess f(X)[U] & = \bP(D \grad f(X)[U]) \\
\end{aligned} \ee
where $D$ is the Euclidean derivative. More detailed information on the related
backgrounds can be found in \cite{opt-manifold-book}.

We next briefly introduce some typical manifolds. 
\begin{itemize}
	\item Sphere $\mathrm{Sp}(n-1):= \{ x \in \R^{n} ~:~ \|x\|_2 = 1 \}$. Let $x(t)$ with $x(0) = x$ be a curve on sphere, i.e., $x(t)^\top x(t) = 1$ for all $t$. Taking the derivatives with respect to $t$, we have 
	\[ \dot{x}(t)^\top x(t) + x(t)^\top \dot{x}(t) = 0. \]
	At $t = 0$, we have $\dot{x}(0)x + x^\top \dot{x}(0) = 0$. Hence, the tangent space is 
	\[ T_x \mathrm{Sp}(n-1) = \{ z~:~z^\top x = 0 \}.  \]
	The projection operator is defined as
	\[ \mathbf{P}_{T_x \mathrm{Sp}(n-1)} (z) = (I - xx^\top )z.  \]
	For a function defined on $\mathrm{Sp}(n-1)$ with respect to the Euclidean metric $g_x(u,v) = u^\top v, \; u,v \in T_{x} \mathrm{Sp}(n-1)$,
	its Riemannian gradient and Hessian at $x$ can be represented by 
	\[ \begin{aligned}
            \grad f(x) & = \mathbf{P}_{T_x \mathrm{Sp}(n-1)}(\nabla f(x)), \\ 
	\hess f(x)[u] & = \mathbf{P}_{T_x \mathrm{Sp}(n-1)}(\nabla^2 f(x)[u] - ux^\top \nabla f(x)), \; u \in T_x \mathrm{Sp}(n-1).
	\end{aligned} \]
	\item Stiefel manifold $\St(n,p):= \{ X \in \R^{n \times p} \,:\, X^\top
	X = I_p \}$. By a similar calculation as the spherical case, we have its
    tangent space:
	\[ T_X \St(n,p) = \{ Z\,:\, Z^\top X + X^\top Z = 0 \}. \]
	The projection operator onto $T_X \St(n,p)$ is
	\[ \mathbf{P}_{T_X \mathrm{St}(n,p)} (Z) = Z - X\sym(X^\top Z), \]
	where $\sym(Z):=(Z + Z^\top)/2$. Given a function defined on $\St(n,p)$ with respect to the Euclidean metric $g_X(U,V) = \tr(U^\top V),\; U,V \in T_X \mathrm{St}(n,p)$, its Riemannian gradient and Hessian at $X$ can be represented by 
	\[ \begin{aligned}
            \grad f(X) & = \mathbf{P}_{T_X \St(n,p)}(\nabla f(X)), \\ 
	\hess f(X)[U] & = \mathbf{P}_{T_X \St(n,p)}(\nabla^2 f(X)[U] - U\sym(X^\top \nabla f(X))), \; U \in T_X \St(n,p).
	\end{aligned} \] 
	\item Oblique manifold $ \mathrm{Ob}(n,p) := \{ X \in \R^{n \times p} \mid \mbox{diag}(X^\top X) = e\}$. Its tangent space is 
	\[ T_X \mathrm{Ob}(n,p) = \{ Z\,:\,\diag(X^\top Z) = 0 \}. \]
	The projection operator onto $T_X \mathrm{Ob}(n,p)$ is
	\[ \mathbf{P}_{T_X \mathrm{Ob}(n,p)} = Z - X \Diag(\diag(X^\top Z)). \]
	Given a function defined on $\mathrm{Ob}(n,p)$ with respect to the Euclidean metric, its Riemannian gradient and Hessian at $X$ can be represented by 
	\[ \begin{aligned}
            \grad f(X) & = \mathbf{P}_{T_X \mathrm{Ob}(n,p)}(\nabla f(X)), \\ 
	\hess f(X)[U] & = \mathbf{P}_{T_X \mathrm{Ob}(n,p)}(\nabla^2 f(X)[U] - U\Diag(\diag(X^\top \nabla f(X)))), 
	\end{aligned} \] 
	with $U \in T_X \mathrm{Ob}(n,p)$.
	\item Grassmann manifold $\mathrm{Grass}(n,p) := \{ \mathrm{span}(X)\,:\, X
        \in \R^{n \times p}, X^\top X = I_p \}$. It denotes the set of all
        $p$-dimensional subspaces of $\R^n$. This manifold is different from
        other manifolds mentioned above. It is a quotient manifold since each
        element is an equivalent class of $n\times p$ matrices. From the definition of $\mathrm{Grass}(p,n)$, the equivalence relation $\sim$ is defined as
	\[ X \sim Y \Leftrightarrow \exists Q \in \R^{p\times p} \mathrm{~with~} Q^\top Q = QQ^\top = I, \; \st \; Y = XQ.  \]
	Its element is of the form 
	\[ [X]:= \{ Y \in \R^{n\times p}: Y^\top Y = I, Y \sim X \}. \]
	Then $\mathrm{Grass}(n,p)$ is a quotient manifold of $\St(n,p)$, i.e.,
    $\St(n,p)/ \sim$. Due to this equivalence, a tangent vector $\xi$ of $T_X
    \mathrm{Grass}(n,p)$ may have many different representations in its
    equivalence class. To find the unique representation, a horizontal space \cite[Section 3.5.8]{opt-manifold-book} is introduced. For a given $X \in \R^{n\times p}$ with $X^\top X = I_p$, the horizontal space is 
	\[ \mathcal{H}_X \mathrm{Grass}(n,p) = \{Z\,:\, Z^\top X = 0 \}. \]  
	Here, a function of the horizontal space is similar to the tangent space when computing the Riemannian gradient and Hessian.
	We have the projection onto the horizontal space
	\[ \mathbf{P}_{\mathcal{H}_X \mathrm{Grass}(n,p)}(Z) = Z -XX^\top Z. \]
	Given a function defined on $\mathrm{Grass}(n,p)$ with respect to the Euclidean metric $g_{X} = \tr(U^\top V), \, U,V \in \mathcal{H}_X \mathrm{Grass}(n,p)$, its Riemannian gradient and Hessian at $X$ can be represented by 
	\[ \begin{aligned}
            \grad f(X) & = \mathbf{P}_{\mathcal{H}_X \mathrm{Grass}(n,p)}(\nabla f(X)), \\ 
	\hess f(X)[U] & = \mathbf{P}_{\mathcal{H}_X \mathrm{Grass}(n,p)}(\nabla^2 f(X)[U] - UX^\top \nabla f(X)), \; U \in T_X \mathrm{Grass}(n,p).
	\end{aligned} \]
	\item Fixed-rank manifold $\Fr(n,p,r) :=\{X \in \R^{ n \times p}\,:\, \rk(X)
        = r \}$ is a set of all $n\times p$ matrices of rank $r$. Using the
        singular value decomposition (SVD), this manifold can be represented equivalently by 
	\[ \Fr(n,p,r) = \{ U\Sigma V^\top\,:\, U \in \St(n,r),\, V\in \St(p,r),\, \Sigma = \diag(\sigma_i) \},  \]
	where $\sigma_1 \geq \cdots \geq \sigma_k > 0$. Its tangent space at $X = U \Sigma V^\top$ is 
	\be \label{eq:tangent-fr} \begin{aligned}
	T_X \Fr(n,p,r) = & \left \{ [U,U_{\bot}] \begin{pmatrix}
	\R^{r\times r} & \R^{r \times (p-r)} \\ \R^{(n-r)\times r} & 0_{(n-r) \times (p-r)} 
	\end{pmatrix} [V, V_\bot]^\top \right \} \\
	= & \{ UMV^\top + U_pV^\top + UV_p^\top\,:\, M \in \R^{r \times r}, \\
	& U_p \in \R^{n\times r}, U_p U = 0, V_p \in \R^{p\times r},\, V_p^\top V = 0  \},
	\end{aligned} \ee
	where $U_\bot$ and $V_{\bot}$ are the orthogonal complements of $U$ and $V$, respectively. 
	The projection operator onto the tangent space is 
	\[ \mathbf{P}_{T_X \Fr(n,p,r)}(Z) = P_U ZP_V + P_U^\bot Z P_V + P_U Z P_V^\bot, \]
	where $P_U = UU^\top$ and $P_U^\bot = I - P_U$. Comparing the representation with \eqref{eq:tangent-fr}, we have 
	\[ M(Z;X) := U^\top Z X,\; U_p(Z;X) = P_U^\bot ZV,\; V_p(Z;X) = P_V^\bot Z^\top U. \]
	Given a function defined on $\Fr(n,p,r)$ with respect to the Euclidean metric $g_X(U,V) = \tr(U^\top V)$, its Riemannian gradient and Hessian at $X = U\Sigma V$ can be represented by 
	\[ \begin{aligned}
            \grad f(X) & = \mathbf{P}_{T_X \Fr(n,p,r)}(\nabla f(X)), \\ 
	\hess f(X)[H] & = U\hat{M}V^\top + \hat{U}_pV^\top + U\hat{V}_p^\top, \; H \in T_X \Fr(n,p,r),
	\end{aligned} \]
	where $\hat{M} = M(\nabla^2 f(X)[H], X), \; \hat{U}_p = U_p(\nabla^2 f(X)[H];X) + P_U^\bot \nabla f(X) V_p(H;X) $ $/\Sigma, \; \hat{V}_p = V_p(\nabla^2 f(X)[H];X) + P_V^\bot  \nabla f(X) U_p(H;X)/\Sigma$. 
	\item The set of symmetric positive definite matrices, i.e., $\mathrm{SPD}(n) =\{ X \in \R^{n\times n} \,:\, X^\top = X, \, X \succ 0 \} $ is a manifold. Its tangent space at $X$ is 
	\[ T_X \mathrm{SPD}(n) = \{ Z: Z^\top = Z \}. \]
	We have the projection onto $T_X \mathrm{SPD}(n)$: 
	\[ \mathbf{P}_{T_X \mathrm{SPD}(n)}(Z) = (Z^\top + Z)/2.  \]
	Given a function defined on $\mathrm{SPD}(n,p)$ with respect to the
    Euclidean metric $g_X(U,V) = \tr(U^\top V), \, U, V\in T_X \mathrm{SPD}(n)$, its Riemannian gradient and Hessian at $X$ can be represented by 
	\[ \begin{aligned}
            \grad f(X) & = \mathbf{P}_{T_X \mathrm{SPD}(n)}(\nabla f(X)), \\ 
	\hess f(X)[U] & = \mathbf{P}_{T_X \mathrm{SPD}(n)}(\nabla^2 f(X)[U]), \; U \in T_X \mathrm{SPD}(n).
	\end{aligned} \]
	\item The set of rank-$r$ symmetric positive semidefinite matrices, i.e., $\mathrm{FrPSD}(n,r)= \{X \in\R^{n\times n}\,:\, X = X^\top, \, X\succeq 0,\, \rk(X) = r \}$. This manifold can be reformulated as
	\[ \mathrm{FrPSD}(n,r) = \{ YY^\top\,:\, Y \in \R^{n\times r}, \rk(Y) = k  \}, \]
	which is a quotient manifold. The horizontal space at $Y$ is 
	\[ T_Y{\mathcal{H}_{\mathrm{FrPSD}(n,r)}} = \{ Z \in \R^{n\times r}\,:\, Z^\top Y = Y^\top Z  \}. \]
	We have the projection operator onto $T_Y{\mathcal{H}_{\mathrm{FrPSD}(n,r)}}$
	\[ \mathbf{P}_{T_Y{\mathcal{H}_{\mathrm{FrPSD}(n,r)}}}(Z) = Z - Y\Omega, \]
	where the skew-symmetric matrix $\Omega$ is the unique solution of the Sylvester equation $\Omega(Y^\top Y) + (Y^\top Y)\Omega = Y^\top Z - Z^\top Y$. Given a function $f$ with respect to the Euclidean metric $g_Y(U,V) = \tr(U^\top V),\, U,V \in T_Y{\mathcal{H}_{\mathrm{FrPSD}(n,r)}}$, its Riemannian gradient and Hessian can be represented by 
	\[ \begin{aligned}
            \grad f(Y) & = \nabla f(Y), \\ 
	\hess f(X)[U] & = \mathbf{P}_{T_Y{\mathcal{H}_{\mathrm{FrPSD}(n,r)}}}(\nabla^2 f(Y)[U]), \; U \in {T_Y{\mathcal{H}_{\mathrm{FrPSD}(n,r)}}}.
	\end{aligned} \]
 	
\end{itemize}

\subsection{Optimality conditions}
We next present the optimality conditions for manifold optimization problem 
in the following form
\be  \label{prob1} \begin{aligned}
	\min_{x \in \M}  \quad & f(x), \\
	\st \quad & c_i(x) = 0, \; i \in \mathcal{E} :=\{ 1, \ldots, \ell \} \\
 & c_i(x) \geq 0, \; i \in \mathcal{I}:=\{ \ell + 1, \ldots, m \}, 
\end{aligned}
\ee
where $\mathcal{E}$ and $\mathcal{I}$ denote the index sets of equality constraints and inequality constraints, respectively, and $ c_i : \M \rightarrow \R, \; i \in \mathcal{E} \cup \mathcal{I}$ are smooth functions on $\M$.  
We mainly adopt the notions in \cite{yang2014optimality}. By keeping the manifold constraint, the Lagrangian function of \eqref{prob1} is 
\[ \mathcal{L}(x,\lambda) = f(x) - \sum_{i\in \mathcal{E} \cup \mathcal{I}} \lambda_i c_i(x), \; x\in \M, \]
where $\lambda_i, \; i\in \mathcal{E} \cup \mathcal{I}$ are the Lagrangian multipliers. Here, we notice that the domain of $\mathcal{L}$ is on the manifold $\M$. Let $\mathcal{A}(x) := \mathcal{E} \cup \{ i \in \mathcal{I} ~:~ c_i(x) = 0 \}$. Then the linear independence constraint qualifications (LICQ) for problem \eqref{prob1} holds at $x$ if and only if 
\[ \grad c_i(x),\; i\in \mathcal{A}(x) \mathrm{~is~linear~independent~on~} T_x \M. \] 
Then the first-order necessary conditions can be described as follows.
\begin{theorem}[First-order necessary optimality conditions (KKT conditions)]
    
    Suppose that $x^*$ is a local minima of \eqref{prob1} and that the LICQ holds at $x^*$, then there exist Lagrangian multipliers $\lambda_i^*, i \in \mathcal{E} \cup \mathcal{I}$ such that the following KKT conditions hold: 
    \be \label{kkt}  \begin{aligned}
    \grad f(x^*) + \sum_{i \in \mathcal{E}\cup \mathcal{I}} \lambda_i^* \grad
    c_i(x^*) & = 0, \\
    c_i(x^*) & = 0, \; \forall i \in \mathcal{E}, \\
    c_i(x^*) \geq 0, \; \lambda_i^* \geq 0, \; \lambda_i^* c_i(x^*) & =0, \; \forall i \in \mathcal{I}. 
    \end{aligned} \ee
\end{theorem}  

Let $x^*$ and $\lambda_i^*, i \in \mathcal{E} \cup \mathcal{I}$ be one of the
solution of the KKT conditions \eqref{kkt}. Similar to the case without the
manifold constraint, we define a critical cone $\mathcal{C}(x^*, \lambda^*)$ as
\[ w \in \mathcal{C}(x^*,\lambda^*) \Leftrightarrow \left\{  \begin{aligned}
w \in T_{x^*}\M,&\\
\iprod{\grad c_i(x^*)}{w} & = 0, \; \forall i \in \mathcal{E}, \\
\iprod{\grad c_i(x^*)}{w} & = 0,\; \forall i \in \mathcal{A}(x^*)\cap
\mathcal{I} \mathrm{~with~} \lambda_i^* >0,\\
\iprod{\grad c_i(x^*)}{w} & \geq 0, \; \forall i \in \mathcal{A}(x^*)\cap \mathcal{I} \mathrm{~with~} \lambda_i^* =0.
\end{aligned} \right. \]
Then we have the following second-order necessary and sufficient conditions. 
\begin{theorem}[Second-order optimality conditions]
	
	\begin{itemize}
		\item Second-order necessary conditions: \\
		Suppose that $x^*$ is a local minima of \eqref{prob1} and the LICQ holds at $x^*$. Let $\lambda^*$ be the multipliers such that the KKT conditions \eqref{kkt} hold. Then we have
		\[ \iprod{\mathrm{Hess}_x \mathcal{L}(x^*,\lambda^*)[w]}{w} \geq 0, \; \forall w \in \mathcal{C}(x^*, \lambda^*), \]
		where $\mathrm{Hess}_x \mathcal{L}(x^*,\lambda^*)$ is the Riemannian Hessian of $\mathcal{L}$ with respect to $x$ at $(x^*,\lambda^*)$.
		\item Second-order sufficient conditions:\\
		Suppose that $x^*$ and $\lambda^*$ satisfy the KKT conditions \eqref{kkt}. If we further have
		\[ \iprod{\mathrm{Hess}_x \mathcal{L}(x^*,\lambda^*)[w]}{w} > 0, \; \forall w \in \mathcal{C}(x^*, \lambda^*), \; w \ne 0, \]
		then $x^*$ is a strict local minima of \eqref{kkt}. 
	\end{itemize}
\end{theorem}

Suppose that we have only the manifold constraint, i.e., $\mathcal{E} \cup
\mathcal{I}$ is empty. For a smooth function $f$ on the manifold $\Mcal$, the
optimality conditions take a similar form to the Euclidean unconstrained case. Specifically, 
if $x^*$ is a first-order stationary point, then it holds that
\[ \grad f(x^*) = 0.\]
If $x^*$ is a second-order stationary point, then
\[ \grad f(x^*) = 0, \quad \hess f(x^*) \succeq 0. \]
If $x^*$ satisfies 
\[ \grad f(x^*) = 0, \quad \hess f(x^*) \succ 0, \]
then $x^*$ is a strict local minimum. For more details, we refer the reader to \cite{yang2014optimality}.

\subsection{First-order type algorithms} \label{subsec:alg-1st}
From the perspective of Euclidean constrained optimization problems, there are
many standard algorithms which can solve this optimization problem on manifold.
However, since the intrinsic structure of manifolds is not considered, these
algorithms may not be effective in practice. By doing curvilinear search along
the geodesic, a globally convergent gradient descent method is proposed in
\cite{gabay1982minimizing}. For Riemannian conjugate gradient (CG) methods
\cite{smith1994optimization}, the parallel translation is used to construct the conjugate directions. Due to the difficulty of calculating geodesics (exponential maps) and parallel translations,  computable retraction and vector transport operators are proposed to approximate the exponential map and the parallel translation \cite{opt-manifold-book}. Therefore, more general Riemannian gradient descent methods and CG methods together with convergence analysis are obtained in \cite{opt-manifold-book}. These algorithms have been successfully applied to various applications \cite{vandereycken2013low, kressner2014low}. Numerical experiments exhibit the advantage of using geometry of the manifold. A proximal Riemannian gradient method is proposed in \cite{hu2018adaptive}. Specifically, the objective function is linearized using the first-order Taylor expansion on manifold and a proximal term is added. The original problem is then transformed into a series of projection problems on the manifold. For general manifolds, the existence and uniqueness of the projection operator can not be guaranteed. But when the given manifold satisfies certain differentiable properties, the projection operator is always locally well-defined and is also a specific retraction operator \cite{absil2012projection}.
 Therefore, in this case, the proximal Riemannian gradient method coincides with
 the Riemannian gradient method. By generalizing the adaptive gradient method in
 \cite{duchi2011adaptive}, an adaptive gradient method on manifold is also
 presented in \cite{hu2018adaptive}. 
 In particular, optimization over Stiefel manifold is an important special case
 of Riemannian optimization. Various efficient retraction operators, vector
 transport operators and Riemannian metric have been investigated to construct more practical gradient descent and CG methods \cite{wen2013feasible,jiang2015framework,zhu2017riemannian}. Non-retraction based first-order methods are also developed in \cite{gao2018new}. 

We next present a brief introduction of first-order algorithms for manifold optimization.
Let us start with the retraction operator $R$. It is a smooth mapping from the tangent bundle $TM := \cup_{x \in \Mcal} \TM $ to $\Mcal$, and satisfies
\begin{itemize}
    \item $R_x(0_x) = x$, $0_x$ is the zero element in the cut space $\TM$,
	\item $ DR_x(0_x)[\xi] = \xi, \; \forall \xi \in \TM$,
\end{itemize}
where $R_x$ is the retraction operator $R$ at $x$. The well-posedness of the retraction operator is shown in Section 4.1.3 of \cite{opt-manifold-book}.
The retraction operator provides an efficient way to pull the points from the
tangent space back onto the manifold. Let $\xi_k \in \TM$ be a descent
direction, i.e., $\iprod{\grad f(x_k)}{\xi_k} < 0$. Another important concept on
manifold is the vector transport operator $\Tcal$. It is a smooth mapping from
the product of tangent bundles $T\M \bigoplus T\M$ to the tangent bundle $T\M$,
and satisfies the following properties.
\begin{itemize}
	\item There exists a retraction $R$ associated with $\Tcal$, i.e., 
	\[ \Tcal_{\eta_x} \xi_x = \frac{d}{dt} R_x(\eta_x + t \xi_x) \mid_{t = 0}.  \]
	\item $\Tcal_{0_x} \xi_x = \xi_x$ for all $x \in \M$ and $\xi_x \in T_x \M$. 
	\item $\Tcal_{\eta_x}(a \xi_x + b \zeta_x) = a \Tcal_{\eta_x} \xi_x + b \Tcal_{\eta_x} \zeta_x$. 
\end{itemize}
The vector transport is a generalization of the parallel translation \cite[Section 5.4]{opt-manifold-book}. 
The general feasible algorithm framework on the manifold can be expressed as
\be \label{eq:scheme} x_{k+1} = R_{x_k} (t_k \xi_k), \ee
where $t_k$ is a well-chosen step size. Similar to the line search method in
Euclidean space, the step size $t_k$ can be obtained by the curvilinear search
on the manifold. Here, we take the Armijo search as an example. Given $\rho, \delta \in (0,1 )$, the monotone and nonmonotone search try to find the smallest integer $h$ to such that
 \begin{align} \label{eq:MLS-Armijo}
f(R_{x_k}( t_k \xi_k) ) &\le f(x_k) + \rho  t_k \iprod {\grad f(x_k)}
{\xi_k}_{x_k},\\
\label{eq:NMLS-Armijo}
f(R_{x_k}( t_k \xi_k) ) &\le C_k + \rho  t_k \iprod {\grad f(x_k)} {\xi_k}_{x_k},
\end{align}
respectively, where $\iprod{\grad f(x_k)}{\xi_k}_{x_k}:= g_{x_k}(\grad f(x_k), \xi_k)$, $t_k = \gamma_k \delta^h$ and $\gamma_k$ is an initial step size.
 The reference
value $C_{k+1}$ is a convex combination of  $C_k$ and
$f( x_{k+1})$ and is calculated via $C_{k+1} = (\varrho Q_k C_k +
f( x_{k+1} ))/Q_{k+1}$,
where $C_0=f(x_0)$, $ Q_{k+1} = \varrho Q_k +1$ and $Q_0=1$.
From the Euclidean optimization, we know that the Barzilai-Borwein (BB) step
size often accelerates the convergence. The BB step size can be generalized to
Riemannian manifold \cite{hu2018adaptive} as 
\be  \label{eq:bb} \gamma_k^{(1)} = \frac{\iprod{s_{k-1}}{s_{k-1}}_{x_k}}{|\iprod{s_{k-1}}{v_{k-1}}_{x_k}|} \quad \mbox{ or } \quad
\gamma_k^{(2)} = \frac{|\iprod{s_{k-1}}{v_{k-1}}_{x_k}|}{
	\iprod{v_{k-1}}{v_{k-1}}_{x_k}}, \ee
where
\bee \label{eq:rbb}  s_{k-1} = - t_{k-1} \cdot {\mathcal T}_{x_{k-1} \rightarrow x_k} ( \grad
f(x_{k-1})), \quad v_{k-1} = \grad f(x_k) + t_{k-1}^{-1} \cdot s_{k-1}, \eee
and ${\mathcal T}_{x_{k-1} \rightarrow x_k}: T_{x_{k-1}} \M \mapsto T_{x_k} \M$ denotes an appropriate vector transport mapping connecting $x_{k-1}$ and $x_k$; see \cite{opt-manifold-book,IanPor17}.	When $\M$ is a submanifold of an
Euclidean space, the Euclidean differences $s_{k-1} = x_{k} - x_{k-1}$ and
$v_{k-1} = \grad f(x_k) - \grad f(x_{k-1})$ are an alternative choice if the
Euclidean inner product is used in \cref{eq:bb}. This choice is often  
attractive since the vector transport is not needed \cite{wen2013feasible,hu2018adaptive}. We note that the differences between first- and second-order algorithms are mainly due to their specific ways of acquiring $\xi_k$. 

In practice, the computational cost and convergence behavior of different retraction operators differ a lot. Similarly, the vector transport plays an important role in CG methods and quasi-Newton methods (we will introduce them later).   
There are many studies on the retraction operators and vector transports. Here,
we take the Stiefel manifold $\St(n,p)$ as an example to introduce several
different retraction operators at the current point $X$ for a given  step size
$\tau$ and descent direction $-D$.
\begin{itemize}
	\item Exponential map \cite{EdelmanAriasSmith1999}
	\bee \label{equ:update:geodesic}
	R_X^{{\mathrm{geo}}}(-\tau D ) =\big[ X, \ Q \big]
	\exp \left(\tau\left[ \begin{array}{cc}-X^{\top} D & \ -R^{\top}\\R & 0\end{array} \right] \right ) \left[\begin{array}{c}I_p\\0\end{array}\right],
	\eee
	where $QR = - (I_n-XX^{\top})D$ is the QR decomposition of $-(I_n -
    XX^{\top})D$. This scheme needs to calculate an exponent of a $2p$-by-$2p$
    matrix	 and an QR decomposition of an $n$-by-$p$ matrix. From
    \cite{EdelmanAriasSmith1999}, an explicit form of parallel translation is unknown. 
	\item Cayley transform \cite{wen2013adaptive}
	\be \label{eq:wenyin}
	R_{X}^{{\mathrm{wy}}}(-\tau D)=X-\tau U \Big(I_{2p}+\frac{\tau}{2}V^{\top}U \Big)^{-1}V^{\top}X,
	\ee
	where $U=[P_XD, \,X]$, $V=[X,\, -P_X D] \in \mathbb{R}^{n \times (2p)}$ with $P_X :=(I-\half XX^\top)$. When $p < n/2$, this scheme is much cheaper than the exponential map. The associated vector transport is \cite{zhu2017riemannian} 
	\[ \Tcal_{\eta_X}^{\mathrm{wy}}(\xi_X) = \left( I - \half W_{\eta_X} \right)^{-1} \left( I + \half W_{\eta_X} \right) \xi_X, \,   W_Z = P_X \eta_X X - X \eta_X P_X, \]	
	\item Polar decomposition \cite{opt-manifold-book}
	\[ R_X^{\mathrm{pd}}(-\tau D) = (X -\tau D)(I_p + \tau^2 D^{\top}D)^{-1/2}. \]
	The computational cost is lower than the Cayley transform, but the Cayley transform gives a better approximation to the exponential map. The associated vector transport is then defined as  \cite{huang2013optimization}
	\[ \Tcal_{\eta_X}^{\mathrm{pd}} \xi_X = Y\Omega + (I - YY^\top) \xi_X (Y^\top(X+\eta_X))^{-1}, \]
	where $Y = R_{X} \eta_X $ and $\mathrm{vec}(\Omega) = (Y^\top(X+\eta_X)) \oplus (Y^\top(X+\eta_X))^{-1} \mathrm{vec}(Y^\top \xi_X - \xi_X^\top Y) $ and $\oplus$ is the Kronecker sum, i.e., $A \oplus B = A \otimes I + I \otimes B$ with Kronecker product $\otimes$. 
	 Numerical experiments show that more iterations may be required compared to the Cayley transform.
	\item QR decomposition
	\[ R_X^{\mathrm{qr}}(-\tau D) = \mbox{qr}(X - \tau D). \]
	It can be seen as an approximation of the polar decomposition. The main cost
    is the QR decomposition of a $n$-by-$p$ matrix. The associated vector transport is defined as \cite[Example 8.1.5]{opt-manifold-book}
	\[ \Tcal_{\eta_X}^{\mathrm{qr}} \xi_X = Y \rho_{skew} (Y^\top \xi_X (Y^\top(X + \eta_X))^{-1}) + (I -YY^\top) \xi_X(Y^\top (X + \eta_X))^{-1}, \]
	where $Y = R_X (\eta_X)$. 
\end{itemize}

The vector transport above requires an associated retraction. Removing the
dependence of the retraction, a new class of
vector transports is introduced in \cite{huang2015broyden}. Specifically, a jointly smooth operator
$\mathcal{L}(x,y) : T_x \M \rightarrow T_y \M$ is defined. In addition, $\mathcal{L}(x,x)$ is required to be an identity for all $x$. For a
$d$-dimensional submanifold $\M$ of $n$-dimensional Euclidean space, two popular
vector transports are defined by the projection \cite[Section 8.1.3]{opt-manifold-book}
\[ \mathcal{L}^{\mathrm{pj}} (x,y) \xi_x = \mathbf{P}_{T_y \M} (\xi_x), \]
and by parallelization \cite{huang2015broyden}
\[ \mathcal{L}^{\mathrm{pl}} (x,y)\xi_x = B_yB_x^\dagger \xi_x, \]
where $B: \mathcal{V} \rightarrow \R^{n \times d}: z \rightarrow B_z$ is a smooth tangent basis field defined on an open neighborhood $\mathcal{V}$ of $\M$ and $B_z^\dagger$ is the pseudo-inverse of $B_z$. With the tangent basis $B_z$, we can also represent the vector transport mentioned above intrinsically, which sometimes reduces computational cost significantly \cite{huang2017intrinsic}. 

To better understand Riemannian first-order algorithms, we present a Riemannian gradient method \cite{hu2018adaptive} in Algorithm \ref{alg:GBB}. One can easily see that the difference to the Euclidean case is an extra retraction step.
\begin{algorithm2e}[t] \label{alg:GBB} \caption{Riemannian gradient method} \label{alg:RLSBB}
	Input $x_0 \in \M$. Set $k=0$, $\gamma_{\min}\in [0,1], 
	\gamma_{\max}\ge 1$, $C_0 = f(x_0),\, Q_0 = 1$.\\
	\While{$\|\grad f(x_k)\| \neq 0$ }
	{
		Compute $ \eta_k = -\grad f(x_k)$. \\
		Calculate $\gamma_k$ according to \eqref{eq:bb} and set $\gamma_k=\max(\gamma_{\min}, \min(\gamma_k,\gamma_{\max} ) )$. Then, compute $C_k, \,Q_k$ and find a step size $t_k$ satisfying
		\eqref{eq:NMLS-Armijo}. \\
		Set $x_{k+1}\gets  R_{x_{k}}( t_k \eta_k)$. \\
		Set $k\gets k+1$.
	}
\end{algorithm2e}

The convergence of Algorithm \ref{alg:GBB} \cite[Theorem 1]{hu2017adaptive} is given as follows.
\begin{theorem} \label{thm:first-order-am}
	Let  $\{x_k\}$ be a sequence generated by Algorithm \ref{alg:RLSBB} using the nonmonotone line search \cref{eq:NMLS-Armijo}. Suppose that $f$ is continuously differentiable on the manifold $\mathcal{M}$. Then, every accumulation point $x_*$ of the sequence $\{x_k\}$ is a stationary point of problem \cref{prob}, i.e., it holds $\grad f(x_*) = 0$.
\end{theorem}
\begin{proof} At first, by using $\iprod{\grad f(x_k)}{\eta_k}_{x_k} = - \|\grad
    f(x_k)\|_{x_k}^2 < 0$ and applying \cite[Lemma 1.1]{ZhaHag04}, we have $f(x_k) \leq C_k$ and $x_k \in \mathcal L$ for all $k \in \N$. Next, due to
	\begin{align*} \lim_{t \downarrow 0} \frac{(f \circ R_{x_k})(t \eta_k) - f(x_k)}{t} - \rho \iprod{\grad f(x_k)}{\eta_k}_{x_k} & \\ &\hspace{-45ex}= \nabla f(R_{x_k}(0))^\top DR_{x_k}(0) \eta_k + \rho \|\grad f(x_k)\|_{x_k}^2 = -(1-\rho) \|\grad f(x_k)\|_{x_k}^2 < 0, \end{align*}
	there always exists a positive step size $t_k \in (0,\gamma_k]$ satisfying the monotone and nonmonotone Armijo conditions \cref{eq:MLS-Armijo} and  \cref{eq:NMLS-Armijo}, respectively. Now, let $x_* \in \mathcal M$ be an arbitrary acccumulation point of $\{x_k\}$ and let $\{x_k\}_K$ be a corresponding subsequence that converges to $x_*$. By the definition of $C_{k+1}$ and \cref{eq:MLS-Armijo}, we have
	\[ C_{k+1} = \frac{\varrho Q_k C_k + f(x_{k+1})}{Q_{k+1}} < \frac{(\varrho Q_k+1) C_k}{Q_{k+1}} = C_k. \]
	Hence, $\{C_k\}$ is monotonically decreasing and converges to some limit $\bar C \in \R \cup \{-\infty\}$. Using $f(x_k) \to f(x_*)$ for $K \ni k \to \infty$, we can infer $\bar C \in \R$ and thus, we obtain
	\[ \infty > C_0 - \bar C = \sum_{k=0}^\infty C_k - C_{k+1} \geq \sum_{k=0}^\infty \frac{\rho t_k \|\grad f(x_k)\|_{x_k}^2}{Q_{k+1}}. \]
	Due to $Q_{k+1} = 1 + \varrho Q_k = 1 + \varrho  + \varrho^2 Q_{k-1} = ... = \sum_{i=0}^{k} \varrho^i < (1-\varrho)^{-1}$, this implies $\{t_k  \|\grad f(x_k)\|_{x_k}^2\} \to 0$. Let us now assume $\|\grad f(x_*)\| \neq 0$. In this case, we have $\{t_k\}_K \to 0$ and consequently, by the construction of Algorithm \ref{alg:RLSBB}, the step size $\delta^{-1} t_k$ does not satisfy \cref{eq:NMLS-Armijo}, i.e., it holds
	\be \label{eq:fst-armijo} - \rho (\delta^{-1}t_k) \|\grad f(x_k)\|_{x_k}^2 < f(R_{x_k}(\delta^{-1}t_k \eta_k)) - C_k \leq f(R_{x_k}(\delta^{-1}t_k \eta_k)) - f(x_k) \ee
	for all $k \in K$ sufficiently large. Since the sequence $\{\eta_k\}_K$ is bounded, the rest of the proof is now identical to the proof of \cite[Theorem 4.3.1]{opt-manifold-book}. In particular, applying the mean value theorem in \cref{eq:fst-armijo} and using the continuity of the Riemannian metric, we can easily derive a contradiction. We refer to \cite{opt-manifold-book} for more details.
\end{proof}

\subsection{Second-order type algorithms} \label{subsec:alg-2nd}
A gradient-type algorithm usually is fast in the early iterations, but it often
slows down or even stagnates when the generated iterations are close to an
optimal solution. When a high accuracy is required, second-order type algorithms
may have its advantage.  

By utilizing the exact Riemannian Hessian and different retraction operators,
Riemannian Newton methods, trust-region methods, adaptive regularized Newton
method have been proposed in
\cite{udriste1994convex,AbsilBakerGallivan2007,opt-manifold-book,hu2018adaptive}.
When the second-order information is not available, the quasi-Newton
type method becomes necessary. As in the Riemannian CG method, we need the vector transport operator to compare different tangent vectors from different tangent spaces. In addition, extra restrictions on the vector transport and the retraction are required for better convergence property or even convergence \cite{qi2011numerical,Ring2012Optimization,seibert2013properties,huang2013optimization,huang2015riemannian,huang2015broyden,huang2018riemannian}. Non-vector-transport based quasi-Newton method is also explored in \cite{hu2018structured}.

\subsubsection{Riemannian trust-region method}
One of the popular second-order algorithms is a Riemannian trust-region (RTR)
algorithm \cite{AbsilBakerGallivan2007,opt-manifold-book}. At the $k$-th
iteration $x_k$, by utilizing the Taylor expansion on manifold, RTR constructs
the following subproblem on the Tangent space:
\be \label{prob:tr-sub}
\min_{\xi \in T_{x_k} \Mcal } \quad m_k(\xi) := \iprod{\grad f(x_k)}{\xi}_{x_k} +\frac{1}{2} \iprod{\hess f(x_k) [\xi] }{\xi}_{x_k}, \quad \st \;\; \|\xi\|_{x_k} \leq \Delta_k,
\ee
where $\Delta_k$ is the trust-region radius. In \cite{NocedalWright06},
extensive methods for solving \eqref{prob:tr-sub} are summarized. Among them,
the Steihaug CG method, also named as truncated CG method, is most popular due
to its good properties and relatively cheap computational cost.  
By solving this trust-region subproblem, we obtain a direction $\xi_k \in
T_{x_k} \M$ satisfying the so-called Cauchy decrease. Then a trial point is computed as $z_k = R_{x_k}(\xi_k)$, where the step size is chosen as $1$. To determine the acceptance of $z_k$, we compute the ratio between the actual reduction and the predicted reduction
\be \label{eq:tr-ratio} \rho_k := \frac{f(x_k) - f(R_{x_k}(\xi_k))}{m_k(0) - m_k(\xi_k)}. \ee 
When $\rho_k$ is greater than some given parameter $0< \eta_1 < 1$, $z_k$ is
accepted.  Otherwise, $z_k$ is rejected. To avoid the algorithm stagnating at some feasible point and promote the efficiency as well, the trust-region radius is also updated based on $\rho_k$. The full algorithm is presented in Algorithm \ref{alg:RTR}. 
\begin{algorithm2e}[t]  \label{alg:RTR}
	\caption{Riemannian trust-region method}	
 	\textbf{Input:} Initial guess $x_0 \in \M$ and parameters $\bar{\Delta}>0, \, \Delta_0 \in (0, \bar{\Delta}), \rho' \in [0, \frac{1}{4})$. \\
	\textbf{Output:} Sequences of iterates $\{x_k\}$ and related information. \\
    \LinesNumbered 
	\For{$k = 0,1,2,\ldots$}{
    Use the truncated CG method to obtain $\xi_k$ by solving
    \eqref{prob:tr-sub}. \\
    Compute the ratio $\rho_k$ in \eqref{eq:tr-ratio}. \\
    \lIf{$\rho_k < \frac{1}{4}$}{ $\Delta_{k+1} = \frac{1}{4} \Delta_k$}  
    \lElseIf{$\rho_k > \frac{3}{4}$ and $\|\xi_k\| = \Delta_k$}{$\Delta_{k+1} = \min(2\Delta_k, \bar{\Delta})$}
    \lElse{$\Delta_{k+1} = \Delta_k$}	
    \lIf{$\rho_k > \rho'$}{$x_{k+1} = R_{x_k}(\xi_k)$}
    \lElse{$x_{k+1} = x_k$}   
    }	
\end{algorithm2e} 

For the global convergence, the following assumptions are necessary for second-order type algorithms on manifold.   
\begin{assumption} \label{assum:2nd-alg}
	(a). The function $f$ is continuous differentiable and bounded from below on the level set $\{x\in\M \,:\, f(x) \leq f(x_0) \}$. \\
    (b). There exists a constant $\beta_{Hess} > 0$ such that 
		\[ \| \Hess f(x_k) \| \leq \beta_{Hess}, \; \forall k = 0,1,2, \ldots \]
\end{assumption}
Algorithm \ref{alg:RTR} also requires a Lipschitz type continuous property on the objective function $f$ \cite[Definition 7.4.1]{opt-manifold-book}.    
\begin{assumption} \label{assum:global-rtr}
	There exists two constants $\beta_{RL} > 0$ and $\delta_{RL} > 0$ such that for all $x\in \M$ and $\xi \in T_x \M \mathrm{~with~} \|\xi\| = 1$, 
	\[ \left| \frac{d}{dt} f\circ R_x(t\xi) \mid_{t = \tau} - \frac{d}{dt} f \circ R_x(t\xi) \mid_{t = 0} \right| \leq \tau\beta_{RL} , \, \forall \tau \leq \delta_{RL}.  \]
%
\end{assumption}
Then the global convergence to a stationary point \cite[Theorem 7.4.2]{opt-manifold-book} is presented as follows.
\begin{theorem}
	Let $\{x_k\}$ be a sequence generated by Algorithm \ref{alg:RTR}. Suppose that Assumptions \ref{assum:2nd-alg} and \ref{assum:global-rtr} holds, then 
	\[ \liminf_{k \rightarrow \infty} \| \grad f(x_k) \| = 0. \]
\end{theorem}
By further assuming the Lipschitz continuous property of the Riemannian gradient
\cite[Definition 7.4.3]{opt-manifold-book} and some isometric property of the
Retraction operator $R$ \cite[Equation (7.25)]{opt-manifold-book}, the
convergence of the whole sequence is proved \cite[Theorem
7.4.4]{opt-manifold-book}. The locally superlinear convergence rate of RTR and
its related assumptions can be found in \cite[Section 7.4.2]{opt-manifold-book}. 

\subsubsection{Adaptive regularized Newton method}
From the perspective of Euclidean approximation, an adaptive regularized Newton
algorithm (ARNT) is proposed for specific and general manifold optimization
problems \cite{wen2013adaptive,wu2015regularized,hu2018adaptive}. In the
subproblem, the objective function is constructed by the second-order Taylor expansion in the Euclidean space and an extra regularization term, while the manifold constraint is kept. Specifically, the mathematical formulation is 
\be \label{prob:arnt} \min_{x \in \Mcal} \quad \hat{m}_k(x): = \iprod{\nabla f(x)}{x - x_k} + \frac{ 1}{2} \iprod{H_k[x - x_k]}{x - x_k} + \frac{\sigma_k}{2} {\|x - x_k\|^2}, \ee
where $H_k$ is the Euclidean Hessian or its approximation. From the definition of Riemannian gradient and Hessian, we have
\be \label{eq:hess-sub}
\begin{aligned}
	\grad \hat{m}_k(x_k) & = \grad f(x_k) \\
	\hess \hat{m}_k(x_k)[U] & =  \bPk(H_k[U]) + \mathfrak{W}_{x_k}(U,\bPk^{\bot}(\nabla f(x_k))) + \sigma_k U,
\end{aligned}  \ee
where $U \in T_{x_k}\Mcal$, $\bPk^{\bot} := I - \bPk$ is the projection onto the normal space and the Weingarten map ${\mathfrak{W}}_x(\cdot,v)$ with $v \in T_{x_k}^{\bot} \M$ is a symmetric linear operator which is related to the second fundamental form of $\M$. To solve \cref{prob:arnt}, a modified CG method is proposed in \cite{hu2018adaptive} to solve the Riemannian Newton equation at $x_k$,
\[ \grad \hat{m}_(x_k) + \hess \hat{m}_k(x_k)[\xi_k] = 0. \]
Since $\hess \hat{m}_k(x_k)$ may not be positive definite, CG may be terminated
if a direction with negative curvature, says $d_k$, is encountered. Different
from the truncated CG method used in RTR, a linear combination of $s_k$ (the output of the truncated CG method) and the negative curvature direction $d_k$ is used to construct a descent direction 
\be \label{eq:subline} \xi_k = \begin{cases} s_k + \tau_k d_k & \text{if } d_k \neq 0, \\ s_k & \text{if } d_k = 0, \end{cases} \quad \text{with} \quad \tau_k := \frac{\iprod{d_k}{\grad m_k(x_k)}_{x_k}}{\iprod{d_k}{\Hess m_k(x_k)[d_k]}_{x_k}}. \ee
A detailed description on the modified CG method is presented in Algorithm \ref{alg:tN}. 
\begin{algorithm2e}[t] \label{alg:tN}
	\caption{A modified CG method for solving subproblem \eqref{prob:arnt} }
    Set $T > 0$, $\theta > 1$, $\epsilon \geq 0$, $\eta_0 = 0$,  $r_0 = \grad
	m_k(x_k)$,  $p_0 = -r_0 $, and  $i = 0$.
	\\
	\While{ $i \leq  n-1$ }
	{	 	
	 Compute $\pi_i = \iprod{p_i}{\Hess \hat{m}_k(x_k)[p_i]}_{x_k}$.\\
		\If {$\pi_i \, /  \iprod{p_i}{p_i}_{x_k} \leq \epsilon$}			
		{
			\lIf {$i = 0$}{set $s_k = -p_0, \, d_k = 0$}
			\lElse{
				set $s_k = \eta_i, $ 
				
				\lIf{$\pi_i \, /  \iprod{p_i}{p_i}_{x_k} \leq -\epsilon$}{$d_k = p_i$, set $\sigma_{est} = |\pi_i| \, /  \iprod{p_i}{p_i}_{x_k} $}
				\lElse{ $d_k = 0$}
				break}}
		
		Set $\alpha_i = \iprod{r_i}{r_i}_{x_k}  / \, \pi_i, \, \eta_{i+1} = \eta_i
		+ \alpha_i p_i$, and $r_{i+1} = r_i + \alpha_i\Hess \hat{m}_k(x_k)[p_i]$.\\
		\If {$\|r_{i+1}\|_{x_k} \leq \min\{\|r_0\|_{x_k}^\theta, T\}$}
		{choose $ s_k = \eta_{i+1}, d_k =0$;
			break;}
		Set $\beta_{i+1} = \iprod{r_{i+1}}{r_{i+1}}_{x_k} /  \iprod{r_i}{r_i}_{x_k}$ and $p_{i+1} = - r_{i+1} + \beta_{i+1} p_i$.\\
		$i\gets i+1$.
	}
	Update $\xi_k$ according to \cref{eq:subline}. 
\end{algorithm2e} 
Then, Armijo search along $\xi_k$ is adopted to obtain a trial point $z_k$. After obtaining $z_k$, we compute the following ratio between the actual reduction and the predicted reduction,
 \be \label{alg:TR-ratio} \rho_k = \frac{ f(z_k) - f(x_k)  } {
	m_k(z_k) }.  \ee
If $\rho_k \ge \eta_1 > 0$, then the iteration is successful and we set
$x_{k+1}= z_k$; otherwise, the iteration is not successful and we set $x_{k+1}=
x_k$, i.e., 
\be \label{eq:update-x}
x_{k+1} = \begin{cases} z_k, & \mbox{ if } \rho_k \ge \eta_1, \\
	x_k, & \mbox{ otherwise}.
\end{cases}
\ee
The regularization parameter $\sigma_{k+1}$ is updated as follows
\be \label{alg:tau-up} \sigma_{k+1} \in \begin{cases} (0,\gamma_0\sigma_k] & \mbox{  if }
	\rho_k \ge \eta_2, \\ [\gamma_0\sigma_k,\gamma_1 \sigma_k] & \mbox{  if } \eta_1 \leq \rho_k
	< \eta_2, \\ [\gamma_1 \sigma_k, \gamma_2 \sigma_k] & \mbox{  otherwise}, \end{cases} \ee
where $0<\eta _1 \le \eta _2 <1 $ and $0 < \gamma_0 < 1<  \gamma _1 \le \gamma _2 $. These parameters determine how aggressively the
regularization parameter is adjusted when an iteration is successful or
unsuccessful. Putting these features together, we obtain Algorithm
\ref{alg:TR-ineact-newton}, which is dubbed as ARNT.
\begin{algorithm2e}[t] 
	\caption{An Adaptive Regularized Newton Method}
	\label{alg:TR-ineact-newton}
	Choose a feasible initial point
	$x_0 \in \M$ and an initial regularization parameter $\sigma_0>0$.
	Choose $0<\eta _1 \le \eta _2 <1 $, $0 < \gamma_0 < 1<  \gamma _1 \le \gamma _2 $.  Set  $k:=0$.
	
	\While{stopping conditions not met}
	{
		
		Compute a new trial point $z_k$ by doing Armijo search along $\xi_k$ obatined by Algorithm \ref{alg:tN}. \\
		Compute the ratio $\rho_k$ via \cref{alg:TR-ratio}. \\
		Update $x_{k+1}$ from the trial point $z_k$ based on \cref{eq:update-x}. \\
		Update $\sigma_{k}$ according to \cref{alg:tau-up}. \\ 
		$k\gets k+1$.
	}	
\end{algorithm2e}

We next present the convergence property of Algorithm \ref{alg:TR-ineact-newton} with the exact Euclidean Hessian (i.e., $H_k = \nabla^2 f(x_k)$)
starting from a few assumptions. 
 \begin{assumption} \label{assum:f} Let $\{x_k\}$ be generated by Algorithm
     \ref{alg:TR-ineact-newton} the exact Euclidean Hessian. 
	\begin{itemize}
		\item[{\rm(A.1)}] The gradient $\nabla f $ is Lipschitz continuous on the convex hull of the manifold $\M$ -- denoted by $\conv(\M)$, i.e., there exists $L_f > 0$ such that
	\end{itemize}
	\[ \|\nabla f(x) - \nabla f(y)\| \leq L_f\| x -y\|, \quad \forall~x,y \in \conv(\M).\]
	\begin{itemize}
		\item[\rm{(A.2)}] There exists $\kg > 0$ such that $\|\nabla f(x_k)\|\leq  \kg$ for all $k \in \N.$
		\item[{\rm(A.3)}] There exists $\kH > 0$ such that $\|\nabla^2 f(x_k)\| \leq \kH$ for all $k \in \N$.
	\end{itemize}
    \begin{itemize}
	\item[{\rm(A.4)}] Suppose there exists $\ubar{\varpi} > 0$, $\obar{\varpi} \geq 1$ such that $\ubar{\varpi}$ and $\obar{\varpi}$ 
		\[ \ubar{\varpi} \| \xi \|_2 \leq g_{x_k}(\xi, \xi) \leq \obar{\varpi} \|\xi \|^2, \; \xi \in T_{x_k} \M,  \]
		 for all $k \in \N$.
	\end{itemize}
\end{assumption}
We note that the assumptions (A.2) and (A.4) hold if $f$ is continuous differentiable and the level set $\{ x\in \Mcal \, :\, f(x) \leq f(x_0)  \}$ is compact. 

The global convergence to an stationary point can be obtained. 
\begin{theorem} \label{thm:convergence} Suppose that Assumptions \ref{assum:2nd-alg} and \ref{assum:f} hold. Then, either
	\[ \grad f(x_\ell) = 0 \ \ \text{for some} \ \ \ell \geq 0  
	\quad \text{or} \quad \liminf_{k \to \infty} \|\grad f(x_k)\|_{x_k} = 0. \] 
\end{theorem}
For the local convergence rate, we make the following assumptions. 

\begin{assumption} \label{assum:local} Let $\{x_k\}$ be generated by Algorithm \ref{alg:TR-ineact-newton}.
	\begin{itemize}
		\item[{\rm(B.1)}] There exists $\beta_R, \delta_R > 0$ such that
		\bee  \left\| \ddt R_x(t \xi) \right \|_{x} \leq \beta_R \eee
		for all $x \in \M$, all $\xi \in T_x \M$ with $\|\xi\|_{x} = 1$ and all $t < \delta_R$.
		\item[{\rm(B.2)}] The sequence $\{x_k\}$ converges to $x_*$. \vspace{0.5ex}
		\item[{\rm(B.3)}] The Euclidean Hessian $\nabla^2 f$ is continuous on $\conv(\M)$. \vspace{0.5ex}
		\item[{\rm(B.4)}] The Riemannian Hessian $\Hess f$ is positive definite at $x_*$ and the constant $\epsilon$ in Algorithm \ref{alg:tN} is set to zero. \vspace{0.5ex}
		\item[{\rm(B.5)}] $H_k$ is a good approximation of the Euclidean Hessian $\nabla^2 f$, i.e., it holds
	\end{itemize}
	%
	\[ \label{eq:hessian} \| H_k - \nabla^2 f(x_k) \| \to 0, \quad \text{whenever} \quad \| \grad f(x_k)\|_{x_k} \to 0. \]
\end{assumption}
Then we have the following results on the local convergence rate.
\begin{theorem} \label{thm:mainthm}
	Suppose that the conditions (B.1)--(B.5) in Assumption \ref{assum:local} are satisfied. Then, the sequence $\{x_k\}$ converges q-superlinearly to $x_*$. 
\end{theorem}
The detailed convergence analysis can be found in \cite{hu2018adaptive}.

\subsubsection{Quasi-Newton type methods}
When the Riemannian Hessian $\hess f(x)$ is computationally expensive or even
not available, quasi-Newton-type methods turn out to be an attractive approach. In literatures \cite{qi2011numerical,Ring2012Optimization,seibert2013properties,huang2013optimization,huang2015riemannian,huang2015broyden,huang2018riemannian}, extensive variants of quasi-Newton methods are proposed. Here, we take the Riemannian Broyden-Fletcher-Goldfarb-Shanno (BFGS) as an example to show the general idea of quasi-Newton methods on Riemannian manifold. Similar to the quasi-Newton method in the Euclidean space, an approximation $\Bcal_{k+1}$ should satisfy the following secant equation 
\[ \Bcal_{k+1} s_{k} = y_{k}, \]
where $s_k = \Tcal_{S_{\alpha_k\xi_k}} {\alpha_k\xi_k}$ and $y_k = \beta_k^{-1} \grad f(x_{k+1}) - \Tcal_{S_{\alpha_k\xi_k}} \grad f(x_k)$ with parameter $\beta_k$. Here, $\alpha_k$ and $\xi_k$ is the step size and the direction used in the $k$-th iteration. $\Tcal_{S_{\alpha_k\xi_k}}$ is an isometric vector transport operator associated with the retraction $R$, i.e., 
\[ \iprod{\Tcal_{S_{\xi_x}} u_x }{\Tcal_{S_{\xi_x}} v_x }_{R_x(\xi_x)} = \iprod{u_x}{v_x}_x. \]
Additionally, $\Tcal_S$ should satisfy the following locking condition,
\[ \Tcal_{S_{\xi_k}} \xi_k = \beta_k \Tcal_{R_{\xi_k}} \xi_k, \; \beta_k =\frac{\|\xi_k\|_{x_k}}{\|\Tcal_{R_{\xi_k}} \xi_k\|_{R_{x_k}(\xi_k)}},  \]
where $\Tcal_{R_{\xi_k}} \xi_k= \frac{d}{dt} R_{x_k}(t\xi_k) \mid_{t=1}$.
Then, the scheme of the Riemannian BFGS is 
\be \label{eq:rbfgs} \Bcal_{k+1} = \hat{\Bcal}_k - \frac{\hat{\Bcal}_ks_k (\hat{\Bcal}_ks_k)^{\flat}}{(\hat{\Bcal}_ks_k)^{\flat} s_k} + \frac{y_k y_k^\flat}{y_k^\flat s_k}, \ee
where $\hat{\Bcal}_k = \Tcal_{S_{\alpha_k\xi_k}} {\alpha_k\xi_k} \circ \Bcal_k \circ \left( \Tcal_{S_{\alpha_k\xi_k}} {\alpha_k\xi_k} \right)^{-1}$ is from $T_{x_{k+1}} \Mcal$ to $T_{x_{k+1}} \Mcal$. With this choice of $\beta_k$ and the isometric property of $\Tcal_S$, we can guarantee the positive definiteness of $\Bcal_{k+1}$. 
After obtaining the new approximation $\Bcal_{k+1}$, the Riemannian BFGS method solves the following linear system 
\[ \Bcal_{k+1} \xi_{k+1} = -\grad f(x_{k+1}) \]
to get $\xi_k$. The detailed algorithm is presented in Algorithm \ref{alg:RBFGS}.
\begin{algorithm2e}[t]  \label{alg:RBFGS}
	\caption{Riemannian BFGS method}	
	\textbf{Input:} Initial guess $x_0 \in \M$, isometric vector transport $\Tcal_S$ associated with the retraction $R$, initial Riemannian Hessian approximation $\Bcal_0:T_{x_0} \M \rightarrow T_{x_0} \M$, which is symmetric positive definite, Wolfe condition parameters $0< c_1 < \half < c_2 < 1$. \\
	\LinesNumbered 
	\For{$k = 0,1,2,\ldots$}{
        Solve $\Bcal_{k} \xi_{k} = -\grad f(x_{k})$ to get $\xi_k$.\\
		Obtain $x_{k+1}$ by doing a Wolfe search along $\xi_k$, i.e., finding
        $\alpha_k >0$ such that the following two conidtions are satisfied
		\[ \begin{aligned}
		f(R_{x_k}(\alpha_k \xi_k )) & \leq f(x_k) + c_1 \alpha_k \iprod{\grad
        f(x_k)}{\xi_k}_{x_k}, \\
		\frac{d}{dt} f(R_{x_k}(t\xi_k)) \mid_{t = \alpha_k}  & \geq c_2
        \frac{d}{dt} f(R_{x_k}(t\xi_k)) \mid_{t=0}.
		\end{aligned}  \] 
		
        Set $x_k = R_{x_k}(\alpha_k \xi_k)$.\\
		Update $\Bcal_{k+1}$ by \eqref{eq:rbfgs}.
	}	
\end{algorithm2e} 
The choice of $\beta_k = 1$ can also guarantee the convergence but with more strict assumptions. One can refer to \cite{huang2015broyden} for the convergence analysis. Since the computation of differentiated retraction may be costly, 
authors in \cite{huang2018riemannian} investigate another way to preserve the positive definiteness of the BFGS scheme. Meanwhile, the Wolfe search is replaced by the Armijo search. As a result, the differentiated retraction can be avoided and the convergence analysis is presented as well.

The aforementioned quasi-Newton methods rely on the vector transport operator. When the vector transport operation is computationally costly, these methods may be less competitive. Noticing the structure of the Riemannian Hessian $\Hess f(x_k)$, i.e.,
\[ \hess f(x_k)[U]  =  \bPk(\nabla^2 f(x_k) [U]) + \mathfrak{W}_{x_k}(U,\bPk^{\bot}(\nabla f(x_k))),\; U \in T_{x_k}\M,  \] 
where the second term $\mathfrak{W}_{x_k}(U,\bPk^{\bot}(\nabla f(x_k)))$ is
often much cheaper than the first term $\bPk(\nabla^2 f(x_k) [U])$. Similar to
the quasi-Newton methods in unconstrained nonlinear least square problems
\cite{kass1990nonlinear} \cite[Chapter 7]{sun2006optimization}, we can focus on the construction of an approximation of the Euclidean Hessian $\nabla^2 f(x_k)$ and use exact formulations of remaining parts. Furthermore, if the Euclidean Hessian itself consists of cheap and expensive parts, i.e., 
\be \label{eq:struct-obj} \nabla^2 f(x_k) = \Hc(x_k) + \He(x_k), \ee
where the computational cost of $\He(x_k)$ is much more expensive than $\Hc(x_k)$, an approximation of $\nabla^2 f(x_k)$ is constructed as
\be \label{eq:struct-app-general} H_k = \Hc(x_k) + C_k, \ee
where $C_k$ is an approximation of $\He(x_k)$ obtained by a quasi-Newton method
in the ambient Euclidean space. If an objective function $f$ is not equipped
with the structure \eqref{eq:struct-obj}, $H_k$ is a quasi-Newton approximation
of $\nabla^2 f(x_k)$. In the construction of the quasi-Newton approximation, a
Nystr\"om approximation technique \cite[Section 2.3]{hu2018structured} is
explored, which turns to be a better choice than the BB type initialization \cite[Chapter 6]{NocedalWright06}. Since the quasi-Newton approximation is constructed in the ambient Euclidean space, the vector transport is not necessary. Then, subproblem \eqref{prob:arnt} is constructed with $H_k$.     
From the expression of the Riemannian Hessian $\hess \hat{m}_k$ in
\eqref{eq:hess-sub}, we see that subproblem \eqref{prob:arnt} gives us a way to
approximate the Riemannian Hessian when an approximation $H_k$ to the Euclidean
Hessian is available. The same procedures of ARNT can be utilized for
\eqref{prob:arnt} with the approximate Euclidean Hessian $H_k$. An adaptive
structured quasi-Newton method given in \cite{hu2018structured} is presented in Algorithm \ref{alg:struct-QN}.
\begin{algorithm2e}[!htbp] 
	\caption{A structured quasi-Newton method} 
	\label{alg:struct-QN}
	Input an initial guess $X^0 \in \Mcal$. 
	Choose $\tau_0>0$, $0<\eta _1 \le \eta _2 <1 $, $1<  \gamma _1 \le \gamma _2 $.  Set $k = 0$.
	\\
	\While{stopping conditions not met}
	{
		Check the structure of $\nabla^2 f(x_k)$ to see if it can be written in a form as \eqref{eq:struct-obj}.\\
		Construct an approximation $H_k$ by utilizing a quasi-Newton method.\\
		Construct and solve the subproblem \eqref{prob:arnt} (by using the modified CG method or the
		Riemannian gradient type method) to
		obtain a new trial point $z_k$.\\
		Compute the ratio $\rho_k$ via \eqref{eq:tr-ratio}. \\
		Update $x_{k+1}$ from the trial point $z_k$ based on \eqref{eq:update-x}. \\
		Update $\tau_{k}$ according to \eqref{alg:tau-up}. \\
		$k\gets k+1$.\\
	}
\end{algorithm2e}

To explain the differences between the two quasi-Newton algorithms more straightforwardly, we take the HF total energy minimization problem \eqref{prob:hf} as an example. From the calculation in \cite{hu2018structured}, we have the Euclidean gradients
\[ \nabla E_{\ks}(X) = {H_{\ks}(X)}X, \quad \nabla E_{\hf}(X) = H_{\hf}(X)X,  \]
where $H_{\ks}(X) := \half L + V_{\ion} + \sum_l \zeta_l w_lw_l^* + {\Diag}((\Re L^\dag) \rho) + {\Diag}(\mu_{\xc}(\rho)^* e)$ and $H_{\hf}(X) = H_{\ks}(X) + \Vcal(XX^*)$. The Euclidean Hessian of $E_{\ks}$ and $E_{\f}$ along a matrix $U \in \C^{n\times p}$ are 
\[ \begin{aligned}
\nabla^2 E_{\ks}(X)[U] & = {H_{\ks} (X)} U + \Diag\left( \big(\Re L^\dag +
\frac{\partial^2 \epsilon_{\xc}}{\partial \rho^2}e\big)(\bar{X} \odot U + X
\odot \bar{U})e\right) X, \\
\nabla^2 E_{\f}(X)[U] & = {\Vcal(XX^*)}U + \Vcal(XU^* + UX^*) X.
\end{aligned}
\]
%
Since $\nabla^2
E_{\f}(X)$ is significantly more expensive than $\nabla^2 E_{\ks}(X)$, we only
need to approximate $\nabla^2 E_{\f}(X)$. The differences $X_{k}-X_{k-1}$, $\nabla E_{\f}(X_k) - \nabla E_{\f} (X_{k-1}) $ are computed. Then, a quasi-Newton approximation $C_k$ of $\nabla^2 E_{\f}$ is obtained without requiring vector transport. By adding the exact formulation of $\nabla^2 E_{\ks}(X_k)$, we have an approximation $H_k$, i.e.,
\[ H_k = \nabla^2 E_{\ks} +C_k.  \]
A Nystr\"om approximation for $C_k$ is also investigated. 
Note that the spectrum of $\nabla^2 E_{\ks}(X)$ dominates the spectrum of
$\nabla^2 E_{\f}(X)$. The structured approximation $H_k$ is more reliable than a
direct quasi-Newton approximate to $\nabla^2 E_{\hf}(X)$ because the spectrum of
$\nabla^2 E_{\ks}$ is inherited from the exact form. 
The Remaining procedure is to  solve subproblem \eqref{prob:arnt} to update $X_k$.   
\subsection{Stochastic algorithms}
For problems arising from machine learning, the objective function $f$ is often a summation of a finite number of functions $f_i, i = 1, \ldots, m$
\[ f(x) = \sum_{i=1}^m f_i(x). \]
For unconstrained situations, there are many efficient algorithms, such as Adam, Adagrad, RMSProp, Adelta, SVRG, etc. One can refer to \cite{lecun2015deep}. For the case with manifold constraints, combining with retraction operators and vector transport operator, these algorithms can be well generalized. However, in the implementation, due to the considerations of the computational costs of different parts, they may have different versions. Riemannian stochastic gradient method is first developed in \cite{bonnabel2013stochastic}. Later, a class of first-order methods and their accelerations are investigated for geodesically convex optimization in \cite{zhang2016first,liu2017accelerated}. With the help of parallel translation or vector transport, Riemannian SVRG methods are generalized in \cite{zhang2016riemannian,sato2017riemannian}. In consideration of the computational cost of the vector transport, non-vector transport based Riemannian SVRG is proposed in \cite{jiang2017vector}.
Since an intrinsic coordinate system is absent, the coordinate-wise update on
manifold should be further investigated. A compromised approach for Riemannian
adaptive optimization methods on product manifolds are presented in
\cite{becigneul2018riemannian}. 

Here, the SVRG algorithm \cite{jiang2017vector} is taken as an example. 
At the current point $X^{s,k}$, we first calculate the full gradient $\mathcal{G}(X^{s,k})$, then randomly sample a subscript from $1$ to $m$ and use this to construct a stochastic gradient with reduced variance as
$\mathcal{G}(X^{s,k}, \xi_{s,k}) = \nabla f(X^{s,0}) + \big( \nabla f_{i_{s,k} }(X^{s,k}) - \nabla f_{i_{s,k}}(X^{s,0}) \big)$, finally move along this direction  with a given step size to next iteration point
\[ X^{s, k+1} = X^{s,k} -\tau_{s} \xi_{s,k} .\]
For Riemannian SVRG \cite{jiang2017vector}, after obtaining the stochastic gradient with reduced Euclidean variance, it first projects this gradient to the tangent space 
\[ \xi_{s,k} = \mathbf{P}_{T_{X^{s,k}} \Mcal}(\mathcal{G}(X^{s,k} )). \]
Then, the following retraction step 
\[ X^{s, k+1} = R_{X^{s,k}}( -\tau_{s} \xi_{s,k} ),\]
is executed to get the next feasible point. The detailed version is outlined in Algorithm \ref{alg:usvrg}. 
 \begin{algorithm2e}[t] \label{alg:usvrg}
 \caption{Riemannian SVRG \cite{jiang2017vector}}
 \For{$s = 0, \ldots, S-1$}{
 calculates the full gradient $\nabla f(X^{s,0})$ and sets the step size $\tau_s > 0$. \\
 \For{$k = 0, \ldots, K-1$}{
  Randomly substitute samples get the subscript $i_{s,k} \subseteq \{1, \ldots, m\}$.
  Calculate a random Euclidean gradient $\mathcal{G}(X^{s,k})$
 \bee \label{equ:sum:procedure:SFO}
 \mathcal{G}(X^{s,k}, \xi_{s,k}) = \nabla f(X^{s,0}) + \big( \nabla f_{i_{s,k} }(X^{s,k}) - \nabla f_{i_{s,k}}(X^{s,0})\big).
 \eee
 Calculate a random Riemann gradient
 \[ \xi_{s,k} = \mathbf{P}_{T_{X^{s,k}} \Mcal}(\mathcal{G}(X^{s,k} )). \]
 Update $X^{s,k+1}$ in the following format
 \bee \label{equ:Xk:update:usvrg}
 X^{s, k+1} = R_{X^{s,k}}( -\tau_{s} \xi_{s,k} ).
 \eee
}
 {Take $X^{s+1,0} \gets X^{s,K}$.}}
 \end{algorithm2e}

\subsection{Algorithms for Riemannian non-smooth optimization}
As shown in \cref{subsec:nonsmooth-spca,subsec:nonsmooth-rlmc,subsec:nonsmooth-scsbd,subsec:cons-nnpca,subsec:cons-kmeans}, many practical problems are with non-smooth objective function and manifold constraints, i.e.,
\[ \min_{x \in \M} \quad f(x):= g(x) + h(x), \]
where $g$ is smooth and $h$ is non-smooth. 
Riemannian subgradient methods \cite{dirr2007nonsmooth,borckmans2014riemannian} are firstly investigated to solve this kind of problems and their convergence analysis is exhibited in \cite{hosseini2015convergence} with the help of Kurdyka-\L ojasiewicz (K\L) inequalities. For locally Lipschitz functions on Riemannian manifolds, a gradient sampling method and a non-smooth Riemannian trust-region method are proposed in \cite{grohs2015nonsmooth,hosseini2017riemannian}. 
Proximal gradient methods on manifold are presented in
\cite{bacak2016second,de2016new}, where the inner subproblem is solved inexactly
by subgradient type methods. The corresponding complexity analysis is given in
\cite{bento2011convergence,bento2017iteration}. Different from the constructions
of the subproblem in \cite{bacak2016second,de2016new}, a more tractable
subproblem without manifold constraints is investigated in
\cite{chen2018proximal} for convex $h(x)$. By utilizing the semi-smooth Newton
method \cite{xiao2018regularized}, the proposed proximal gradient method on
manifold enjoys a faster convergence.   Another class of methods is based on
operator-splitting techniques. Some variants of the alternating direction method
of multipliers (ADMM) are studied in
\cite{lai2014splitting,kovnatsky2016madmm,wang2019global,zhang2017primal,birgin2018augmented,liu2019simple}.

We briefly introduce the proximal gradient method on manifold \cite{chen2018proximal} here. Assume that the convex function $h$ is Lipschitz continuous. At each iteration $x_k$, the following subproblem is constructed
\be \label{prob:manpg} \begin{aligned}
	\min_d \quad & \iprod{\grad g(x_k)}{d} + \frac{1}{2t} \|d\|_F^2 + h(x_k + d) \\
	\st \quad & d \in T_{x_k} \M,
\end{aligned} \ee  
where $t > 0$ is a step size. Given a retraction $R$, problem \eqref{prob:manpg} can be seen as a first-order approximation of $f(R_{x_k}(d))$ near the zero element $0_{x_k}$ on $T_{x_k} \M$. Specifically, it follows from the definition of the Riemannian gradient that $\grad g(x_k) = \nabla g(R_{x_k}(0))$. From the Lipschitz continuous property of $h$ and the definition of $R$, we have 
\[ | h(R_{x_k}(d)) - h(x_k+d) | \leq L_h \| R_{x_k}(d) - (x_k +d) \|_F = O(\|d\|_F^2),  \]
where $L_h$ is the Lipschitz constant of $h$. Therefore, we conclude
\[ f(R_{x_k}(d)) = \iprod{\grad g(x_k)}{d} + h(x_k + d) + O(\|d\|_F^2),\;\; d \rightarrow 0.   \]
Then the next step is to solve \eqref{prob:manpg}. Since \eqref{prob:manpg} is
convex and with linear constraints, the KKT conditions are sufficient and
necessary for the global optimality. Specifically, we have 
\[ d(\lambda) = \mathrm{prox}_{th}(b(\lambda)) -x_k, \; \; b(\lambda) = x_k - t(\grad f(x_k) - \mathcal{A}_k^*(\lambda)), \;\; \mathcal{A}_k(d(\lambda)) = 0,  \]
where $d \in T_{x_k} \M$ is represented by $\mathcal{A}_k(d) = 0$ with a linear operator $\mathcal{A}_k$, $\mathcal{A}_k^*$ is the adjoint operator of $\mathcal{A}_k$. Define $E(\lambda) := \mathcal{A}_k(d(\lambda))$, it is proved in \cite{chen2018proximal} that $E$ is monotone and then the semi-smooth Newton method in \cite{xiao2018regularized} is utilized to solve the nonlinear equation $E(\lambda) = 0$ to obtain a direction $d_k$. Combining with a curvilinear search along $d_k$ with $R_{x_k}$, the decrease on $f$ is guaranteed and the global convergence is established. 
%

\subsection{Complexity Analysis}
The complexity analysis of the Riemannian gradient method and the Riemannian
trust region method has been studied in \cite{boumal2016global}. Similar to the
Euclidean unconstrained optimization, the Riemannian gradient method (using a
fixed step size or Armijo curvilinear search) converges to $\|\grad f(x)\| \leq
\varepsilon$ up to $O(1/\varepsilon^2)$ steps. Under mild assumptions, a
modified Riemannian trust region method converges to $\|\grad f(x)\| \leq
\varepsilon, \; \hess f(x) \succeq - \sqrt{\varepsilon} I$ at most $O(\max\{
1/\varepsilon^ {1.5}, 1/ \varepsilon^{2.5} \})$ iterations. For objective
functions with multi-block convex but non-smooth terms, an ADMM of complexity of $O(1/\varepsilon^4)$ is proposed in \cite{zhang2017primal}.
For the cubic regularization methods on the Riemannian manifold, recent studies \cite{zhang2018cubic,agarwal2018adaptive} show a convergence to $\|\grad f(x)\| \leq \varepsilon, \; \hess f(x) \succeq - \sqrt{\varepsilon} I$ with complexity of $O(1/\varepsilon^{1.5 })$.

\section{Analysis for manifold optimization} \label{sec:analysis}

\subsection{Geodesic convexity}
For a convex function in the Euclidean space, any local minima is also a global minima.
An interesting extension is the geodesic convexity of functions. Specifically, a function defined on manifold is said to be geodesically convex if it is convex along any geodesic.
 Similarly, a local minima of a geodesically convex function on manifold is also a global minima. 
Naturally, a question is how to distinguish the geodesically convex function. 
\begin{definition}
	Given a Riemannian manifold $(\Mcal, g)$, a set $ \mathcal{K} \subset \Mcal$ is called $g$-fully geodesic, if for any $p,q \in \mathcal{ K}$, any geodesic $\gamma_{pq}$ is located entirely in $\mathcal{K}$.
\end{definition}
For example, $D_c := \{ P \in \mathbb{S}^n_{++} ~|~ \det(P) = c \}$ with a
positive constant $c$ is not a convex set in $\R^{n \times n}$, but is a fully geodesic set of Riemannian manifolds ($\mathbb{S}^n_{++}, g$), where the Riemannian metric $g$ at $P$ is $g_P(U,V) := \tr(P^{-1}UP^{-1}V)$.
Now we present the definition of the $g$-geodesically convex function.
\begin{definition}
	Given a Riemannian manifold $(\Mcal, g)$ and a $g$-fully geodesic set $
    \mathcal{K} \subset \Mcal$, a function $f : \mathcal{K} \rightarrow \R$ is
    $g$-geodesically convex if for any $p,q \in \mathcal{K}$ and any geodesic
    $\gamma_{pq}\,:\,[0,1] \rightarrow \mathcal{K}$ connecting $p,q$, it holds: 
	\[  f(\gamma_{pq}(t)) \leq (1-t) f(p) + tf(q), \; \forall t \in [0,1]. \]
\end{definition}
A $g$-fully geodesically convex function may not be convex. For example, $f(x):= (\log x)^2, \, x\in \R_+$ is not convex in the Euclidean space, but is convex with respect to the manifold ($\R_+, g$), where $g_x(u,v):= ux^{-1}v$.

Therefore, for a specific function, it is of significant importance to define a proper Riemannian metric to recognize the geodesic convexity. A natural problem is, given a manifold $\Mcal$ and a smooth function $f: \Mcal \rightarrow \R$, whether there is a metric $g$ such that $f$ is geodesic convex with respective to $g$? It is generally not easy to prove the existence of such a metric. From the definition of the geodesic convexity, we know that if a function has a non-global local minimum, then this function is not geodesically convex for any metric. For more information on geodesic convexity, we refer to \cite{vishnoi2018geodesic}.

\subsection{Convergence of self-consistent field iterations}
In \cite{liu2014convergence,liu2015analysis}, several classical theoretical
problems from KSDFT are studied. Under certain conditions, the equivalence
between KS energy minimization problems and KS equations are established. In
addition, a lower bound of non-zero elements of the charge density is also
analyzed. By treating the KS equation as a fixed point equation with respect to
a potential function, the Jacobian matrix is explicitly derived using the
spectral operator theory and the theoretical properties of the SCF method are
analyzed. It is proved that the second-order derivatives of the
exchange-correlation energy are uniformly bounded if the Hamiltonian has a
sufficiently large eigenvalue gap. Moreover, SCF converges from any initial
point and enjoys a local linear convergence rate. Related results can be found in \cite{dai2017conjugate,zhu2017riemannian, cai2017eigenvector, zhao2015riemannian, bai2018robust, zhang2015maximization}.

Specifically, for the KS equation \eqref{eq:ks}, we define the potential function
\be \label{eq:cV}
V:= \mathbb{V}(\rho) = L^\dagger \rho + \mu_{xc}(\rho)^\top e
\ee
and 
\be \label{eq:H-V} H(V):= \half L + \sum_l \zeta_l w_lw_l^* + V_{ion} +
\diag(V). \ee
Then, we have $H_\ks(\rho)= H(V(\rho))$. From \eqref{eq:ks}, $X$ are the
eigenvectors corresponding to the $p$-smallest eigenvalues of $H(V)$, which is
dependent on $V$. Then, a fixed point mapping for $V$ can be written as
\be\label{eq:V}
V = \mathbb{V}( F_\phi(V)),
\ee
where $F_\phi(V) = \diag(X(V)X(V)^\top)$.
Therefore, each iteration of SCF is to update $V_k$ as 
\be\label{eq:V-SCF}V_{k+1} = \mathbb{V}(F_\phi(V_k)).
\ee
For SCF with a simple charge-mixing strategy, the update scheme can be written as
\be \label{eq:V-simple-mixing}
V_{k+1} = V_k - \alpha (V_k - \mathbb{V}(F_\phi(V_k))),
\ee
where $\alpha$ is an appropriate step size. Under some mild assumptions, SCF converges with a local linear convergence rate. 
\begin{theorem}
	Suppose that $\lambda_{p+1}(H(V)) - \lambda_{p}(H(V)) > \delta, \; \forall
    V$, the second order derivatives of $\epsilon_\xc$ are upper bounded and there is a constant $\theta$ such that $\| L^\dagger +\frac{\partial \mu_{xc}(\rho)}{\partial \rho} e\|_2 \leq \theta,\; \forall \rho \in \R^n $. Let $b_1:= 1 - \frac {\theta}{\delta} > 0$,
	$\{V_k\}$ be a sequence generated by \eqref{eq:V-simple-mixing} with a step size of $\alpha$ satisfying
	\bee \label{eq:gl-conv-alpha}
	0<\alpha<\frac{2}{2-b_1}.
	\eee
	Then, $\{V^i\}$ converges to a solution of the KS equation \eqref{eq:ks}, and its convergence rate is not worse than $|1-\alpha| + \alpha (1-b_1) $.
\end{theorem}

\subsection{Pursuing global optimality}
In the Euclidean space, a common way to escape the local minimum is to add white noise to the gradient flow, which leads to a stochastic differential equation
\bee
\mathrm{d} X(t) = -\nabla f(X(t)) \mathrm{d} t + \sigma(t) \mathrm{d} B(t),
\label{eq:sde:rn}
\eee
where $B(t)$ is the standard $n$-by-$p$ Brownian motion.  A generalized noisy
gradient flow on the Stiefel manifold is investigated in \cite{yuan2017global}
\bee
\mathrm{d} X(t) = - \grad f(X(t)) \mathrm{d} t + \sigma(t) \circ \mathrm{d} B_{\Mcal}(t),
\label{eq:sde:stfmfd}
\eee
where $B_{\Mcal}(t)$ is the Brownian motion on the manifold $\Mcal:= \St(n,p)$.
The construction of a Brownian motion is then given in an extrinsic form. Theoretically, it can converge to the global minima by assuming second-order continuity.

\subsection{Community detection}
For community detection problems, a commonly used model is called the degree-correlated stochastic block model (DCSBM). It assumes that there are no overlaps between nodes in different communities. Specifically, the hypothesis node set $[n] =
\{1,...,n\}$ contains $k$ communities, $\{ C_1^*, \ldots, C_k^* \}$ satisfying
$$C_a^*\cap C_b^* = \emptyset,\forall a\neq b \text{ and} \cup_{a=1}^kC_a^* = [n].$$
In DCSBM, the network is a random graph, which can be represented by a matrix with all elements 0 to 1 represented by $B\in \mathbb{S}^{k}$.
Let $A\in\{0,1\}^{n\times n}$ be the adjacency matrix of this network and
$A_{ii}=0, \forall i\in[n]$. Then for $i\in C_a^*, j\in C_b^*, i\neq j$,
$$A_{ij} = \begin{cases}
1, & \mbox{ with probability } B_{ab}\theta_i\theta_j, \\
0, & \mbox{ with probability } 1-B_{ab}\theta_i\theta_j,
\end{cases}$$
where the heterogeneity of nodes is characterized by the vector $\theta$. More specifically, larger $\theta_i$ corresponds to $i$ with more edges connecting other nodes.
For DCSBM, the aforementioned relaxation model \eqref{prob:comm-relax} is proposed in \cite{zhang2017sparse}. By solving \eqref{prob:comm-relax}, an approximation of the global optimal solution can be obtained with high probability.
\begin{theorem}
	Define
	$ G_a=\sum_{i\in C_a^*}{\theta_i}, H_a=\sum_{b=1}^kB_{ab}G_b,
	F_i=H_a\theta_i.$ Let $U^*$ and $\Phi^*$ be global optimal solutions for \eqref{prob:comm-relax} and \eqref{prob:comm}, respectively and define $\Delta = U^*(U^*)^{\top}-\Phi^*(\Phi^*)^{\top}$.
	Suppose that
	$\max_{1\leq a<b\leq k}\frac{B_{ab}+\delta}{H_aH_b}<\lambda<\min_{1\leq
		a\leq k}\frac{B_{aa}-\delta}{H_a^2}$
	(where $\delta>0$). Then, with high probability, we have
	$$\|\Delta\|_{1,\theta}\leq \frac{C_0}{\delta}\left(1+\left(\max_{1\leq a\leq k}\frac{B_{ Aa}}{H_a^2}\|f\|_1\right)\right)(\sqrt{n\|f\|_1}+n).$$
\end{theorem}

\subsection{Max cut}
Consider the SDP relaxation \eqref{prob:maxcut-sdp} and the non-convex relaxation problem with low rank constraints \eqref{prob:maxcut-noncvxrelax}. If $p \geq \sqrt{2n}$, the composition of a solution $V^*$ of \eqref{prob:maxcut-noncvxrelax}, i.e., $V^*(V^*)^\top$, is always an optimal solution of SDP \eqref{prob:maxcut-sdp} \cite{barvinok1995problems, pataki1998rank, burer2003nonlinear}.
If $p \geq \sqrt{2n}$, for almost all matrices $C$,
problem \eqref{prob:maxcut-noncvxrelax} has a unique local minimum and this minimum is also a global minimum of the original problem
\eqref{prob:maxcut} \cite{boumal2016non}. The relationship between solutions of the two problems \eqref{prob:maxcut-sdp} and \eqref{prob:maxcut-noncvxrelax} is presented in \cite{mei2017solving}. Define $
\mathrm{SDP}(C) = \max\{ \langle C, X\rangle : X \succeq 0, X_{ii} = 1, i \in [n] \}$. A point $V \in \mathrm{Ob}(p,n)$ is called an $\varepsilon$-approximate concave point of \eqref{prob:maxcut-noncvxrelax}, if 
\[ \iprod{U}{\hess f(V)[U]} \leq \varepsilon \|u\|^2, \;\; \forall U \in T_{V} \mathrm{Ob}(p,n). \]
The following theorem tells the approximation quality of an $\varepsilon$-approximate concave point of \eqref{prob:maxcut-noncvxrelax}.  
\begin{theorem}
	\label{thm:apprxconcavepoint}
	For any $\varepsilon$-approximate concave point $V$ of \eqref{prob:maxcut-noncvxrelax}, we have
	\begin{align}
	\tr(CV^\top V) \geq \mathrm{SDP}(C) - \frac{1}{p-1} (\mathrm{SDP}(C) + \mathrm{SDP}(-C)) - \frac{n}{2} \varepsilon.
	\end{align}
\end{theorem}
Another problem with similar applications is the $\mathbb{Z}_2$ synchronization problem \cite{bandeira2016low}. Specifically, given noisy observations $Y_{ij} = z_iz_j + \sigma W_{ij}$, where $W_{i>j} \sim \mathcal{N}(0,1) $ and $W_{ij} = W_{ji}, W_{ii} = 0$, we want to estimate the unknown labels $z_i \in \{\pm 1\}$.
It can be seen as a special case of the max cut problem with $p=2$. The following results are presented in \cite{bandeira2016low}. 
\begin{theorem}
	If $\sigma < \frac{1}{8} \sqrt{n}$, then,
	with a high probability, all second-order stable points $Q$ of problem
    \eqref{prob:maxcut-noncvxrelax} ($p=2$) have the following non-trivial relationship with the true label $z$, i.e., for each such $\sigma$, there is $\varepsilon$ such that
	\[ \frac{1}{n} \|Q^\top z\|_2 \geq \varepsilon.\]
\end{theorem}

\subsection{Burer-Monteiro factorizations of smooth semidefinite programs} \label{subsec:sdp}
Consider the following SDP
\be \label{prob:sdp} \min_{X \in\mathbb{S}^{n\times n}} \tr(C X) \quad \st \; \; \mathcal{A}(X)= b,\, X \succeq 0, \ee
where $C \in \mathbb{S}^{n\times n}$ is a cost matrix,  $\mathcal{A}: \mathbb{S}^{n\times n} \rightarrow \R^m$ is a linear operator and $\mathcal{A}(X) = b$  leads to $m$ equality constraints on $X$, i.e, $\tr(A_iX) = b_i \mathrm{~with~} A_i \in \mathbb{S}^{n\times n},\, b \in \R^m, \, i=1,\ldots, m$. Define $\mathcal{C}$ as the constraint set
\[ \mathcal{C} = \{X \in \mathbb{S}^{n\times n}:\mathcal{A}(X) = b,\, X \succeq 0 \}. \]
If $\mathcal{C}$ is compact, it is proved in \cite{barvinok1995problems,pataki1998rank}  that \eqref{prob:sdp} has a global minimum of rank $r$ with $\frac{r(r+1)}{2} \leq m$. This allows to use the Burer-Monteiro factorizations \cite{burer2003nonlinear} (i.e., let $X = YY^\top$ with $Y\in \R^{n\times p},\, \frac{p(p+1)}{2} \geq m$) to solve the following non-convex optimization problem
\be \label{prob:sdp-ncvx}
\min_{Y \in \R^{n\times p}} \tr(CYY^\top) \quad \st \;\; \mathcal{A}(YY^\top) = b. 
\ee
Here, we define the constraint set 
\be \label{eq:smooth-sdp} \Mcal = \M_p = \{Y \in \R^{n\times p} \,:\, \mathcal{A}(YY^\top) = b \} .\ee
Since $\Mcal$ is non-convex, there may exist many non-global local minima of \eqref{prob:sdp-ncvx}. It is claimed in \cite{burer2005local} that each local minimum of \eqref{prob:sdp-ncvx} maps to a global minimum of \eqref{prob:sdp} if $\frac{p(p+1)}{2} > m$. By utilizing the optimality theory of manifold optimization, any second-order stationary point can be mapped to a global minimum of \eqref{prob:sdp} under mild assumptions \cite{boumal2018deterministic}. Note that \eqref{eq:smooth-sdp} is generally not a manifold. When the dimension of the space spanned by $\{A_1Y, \ldots, A_mY\}$, denoted by $\rk{\mathcal{A}}$, is fixed for all $Y$, $\M_p$ defines a Riemannian manifold. Hence, we need the following assumptions. 
\begin{assumption} \label{assum:sdp}
	
	For a given $p$ such that $\Mcal_p$ is not empty, assume at least one of the
    following conditions are satisfied.
	\begin{itemize}
		\item[\rm(SDP.1)] $\{A_1Y, \ldots, A_mY\}$ are linearly independent in $\R^{n\times p}$ for all $Y\in M_p$
		\item[\rm(SDP.2)] $\{A_1Y, \ldots, A_mY\}$ span a subspace of constant dimension in $\R^{n\times p}$
		 for all $Y $ in an open neighborhood of $\M_p \in \R^{n \times p}$. 
	\end{itemize}
\end{assumption}
By comparing the optimality conditions of \eqref{prob:sdp-ncvx} and the KKT conditions of \eqref{prob:sdp}, the following equivalence between \eqref{prob:sdp} and \eqref{prob:sdp-ncvx} is established in \cite[Theorem 1.4]{boumal2018deterministic}.
\begin{theorem}
	Let $p$ be such that $\frac{p(p+1)}{2} > \rk{\mathcal{A}}$. Suppose that Assumption \ref{assum:sdp} holds. For almost any cost matrix $C \in \mathbb{S}^{n\times n}$, if $Y \in \M_p$ satisfies first- and second-order necessary
	optimality conditions for \eqref{prob:sdp-ncvx}, then $Y$ is globally optimal and $X = YY^\top$ is globally optimal
	for \eqref{prob:sdp}.
\end{theorem}
 
\subsection{Little Grothendieck problem with orthogonality constraints}
Given a positive semidefinite matrix $C \in \R^{dn \times dn}$, the little Grothendieck problem with orthogonality constraints can be expressed as
\be \label{prob:groth} \max_{O_1,\ldots, O_d \in \mathcal{O}_d} \sum_{i=1}^n \sum_{j=1}^n \tr(C_{ij}^\top O_i O_j^\top), \ee
where $C_{ij}$ represents the $(i,j)$-th $d \times d$ block of $C$, $\mathcal{O}_d$ is a group of $d \times d$ orthogonal matrices (i.e., $O \in \mathcal{O}_d$ if and only if $O^\top O = OO^\top = I$.) A SDP relaxation of \eqref{prob:groth} is as follows \cite {bandeira2016approximating}
\be
\max_{\substack{ G\in\mathbb{R}^{dn\times dn} \\ G_{ii}=I_{d\times d} , \ G\succeq 0 }} \tr(CG).
\label{eq_OCSDPm}
\ee
For the original problem \eqref{prob:groth}, a randomized approximation algorithm is presented in \cite{bandeira2016approximating}. Specifically, it consists of the following two procedures. 
\begin{itemize}
	\item Let $G$ be a solution to problem \eqref{eq_OCSDPm}. Denote by the Cholesky decomposition $G = LL^\top$. Let $X_i$ be a $d\times (nd)$ matrix such that $L = (X_1^\top,X_2^\top,\ldots, X_n^\top)^\top$.  
	\item Let $R \in \R^{(nd)\times d}$ be a real-valued Gaussian random matrix whose entries are i.i.d.$\Ncal(0,\frac1d)$.
	The approximate solution of the problem \eqref{prob:groth} can be calculated as follows
	\[
	V_i = \mathcal{P}(X_i R),
	\]
	where $\mathcal{P}(Y)=\argmin_{Z\in \mathcal{O}_d}\|Z-Y\|_F$ with $Y \in \R^{d \times d}$. 
\end{itemize}
For the solution obtained in the above way, a constant approximation ratio on the objective function value is shown, which recovers the known $\frac{2}{\pi}$ approximation guarantee for the classical little Grothendieck problem. 
\begin{theorem}\label{lemma_maintheorem}
	Given a symmetric matrix $C\succeq 0$. Let $V_1,\dots,V_n\in \mathcal{O}_d$ be obtained as above. Then
	\[
	\E \left[ \sum_{i=1}^n\sum_{j=1}^n\tr\left(C_{ij}^TV_iV_j^T\right) \right] \geq \alpha(d)^ 2 \max_{O_1,.. .,O_n \in \mathcal{O}_d} \sum_{i=1}^n\sum_{j=1}^n\tr\left(C_{ij}^TO_iO_j^ T\right),
	\]
	where 
	\[ \alpha(d) := \E \left[ \frac1d \sum_{j=1}^d \sigma_j(Z)\right], \]
	$Z \in \R^{d \times d}$ is a Gaussian random matrix whose components i.i.d. $\Ncal(0,\frac1d)$ and $\sigma_j(Z)$ is the $j$-th singular value of $Z$.
\end{theorem}

\section{Conclusions}
Manifold optimization has been extensively studied in literatures. We review the
definition of manifold optimization, a few related applications, algorithms and
analysis. However, there are still many issues and challenges.  Many
manifold optimization problems that can be effectively solved are still limited
 to relatively simple structures such as orthogonal constraints, rank
constraints and etc. For other manifolds with complicated structures, what are
the most efficient choices of Riemannian metrics and retraction operators are
 not obvious. Another interesting topic is to combine the manifold structure with the characteristics of specific problems and applications, such as graph-based data analysis, real-time data flow analysis, biomedical image analysis, etc. 
 Nonsmooth problems appear to be more and more attractive.  

\bibliographystyle{siamplain}
\bibliography{optimization}

\begin{thebibliography}{100}

\bibitem{AbsilBakerGallivan2007}
{\sc P.-A. Absil, C.~G. Baker, and K.~A. Gallivan}, {\em Trust-region methods
  on {R}iemannian manifolds}, Found. Comput. Math., 7 (2007), pp.~303--330.

\bibitem{opt-manifold-book}
{\sc P.-A. Absil, R.~Mahony, and R.~Sepulchre}, {\em Optimization algorithms on
  matrix manifolds}, Princeton University Press, Princeton, NJ, 2008.

\bibitem{absil2012projection}
{\sc P.-A. Absil and J.~Malick}, {\em Projection-like retractions on matrix
  manifolds}, SIAM Journal on Optimization, 22 (2012), pp.~135--158.

\bibitem{agarwal2018adaptive}
{\sc N.~Agarwal, N.~Boumal, B.~Bullins, and C.~Cartis}, {\em Adaptive
  regularization with cubics on manifolds with a first-order analysis}, arXiv
  preprint arXiv:1806.00065,  (2018).

\bibitem{bacak2016second}
{\sc M.~Bac{\'a}k, R.~Bergmann, G.~Steidl, and A.~Weinmann}, {\em A second
  order nonsmooth variational model for restoring manifold-valued images}, SIAM
  Journal on Scientific Computing, 38 (2016), pp.~A567--A597.

\bibitem{bai2018robust}
{\sc Z.~Bai, D.~Lu, and B.~Vandereycken}, {\em Robust {R}ayleigh quotient
  minimization and nonlinear eigenvalue problems}, SIAM J. Sci. Comput., 40
  (2018), pp.~A3495--A3522.

\bibitem{bandeira2016low}
{\sc A.~S. Bandeira, N.~Boumal, and V.~Voroninski}, {\em On the low-rank
  approach for semidefinite programs arising in synchronization and community
  detection}, in Conference on Learning Theory, 2016, pp.~361--382.

\bibitem{bandeira2016approximating}
{\sc A.~S. Bandeira, C.~Kennedy, and A.~Singer}, {\em Approximating the little
  grothendieck problem over the orthogonal and unitary groups}, Mathematical
  programming, 160 (2016), pp.~433--475.

\bibitem{barvinok1995problems}
{\sc A.~I. Barvinok}, {\em Problems of distance geometry and convex properties
  of quadratic maps}, Discrete \& Computational Geometry, 13 (1995),
  pp.~189--202.

\bibitem{becigneul2018riemannian}
{\sc G.~B{\'e}cigneul and O.-E. Ganea}, {\em {R}iemannian adaptive optimization
  methods}, arXiv preprint arXiv:1810.00760,  (2018).

\bibitem{bento2011convergence}
{\sc G.~Bento, J.~Neto, and P.~Oliveira}, {\em Convergence of inexact descent
  methods for nonconvex optimization on {R}iemannian manifolds}, arXiv preprint
  arXiv:1103.4828,  (2011).

\bibitem{bento2017iteration}
{\sc G.~C. Bento, O.~P. Ferreira, and J.~G. Melo}, {\em Iteration-complexity of
  gradient, subgradient and proximal point methods on {R}iemannian manifolds},
  Journal of Optimization Theory and Applications, 173 (2017), pp.~548--562.

\bibitem{birgin2018augmented}
{\sc E.~G. Birgin, G.~Haeser, and A.~Ramos}, {\em Augmented lagrangians with
  constrained subproblems and convergence to second-order stationary points},
  Computational Optimization and Applications, 69 (2018), pp.~51--75.

\bibitem{bonnabel2013stochastic}
{\sc S.~Bonnabel}, {\em Stochastic gradient descent on {R}iemannian manifolds},
  IEEE Transactions on Automatic Control, 58 (2013), pp.~2217--2229.

\bibitem{borckmans2014riemannian}
{\sc P.~B. Borckmans, S.~E. Selvan, N.~Boumal, and P.-A. Absil}, {\em A
  {R}iemannian subgradient algorithm for economic dispatch with valve-point
  effect}, Journal of Computational and Applied Mathematics, 255 (2014),
  pp.~848--866.

\bibitem{boumal2016global}
{\sc N.~Boumal, P.-A. Absil, and C.~Cartis}, {\em Global rates of convergence
  for nonconvex optimization on manifolds}, IMA J. Numer. Anal.,  (2016).

\bibitem{boumal2016non}
{\sc N.~Boumal, V.~Voroninski, and A.~Bandeira}, {\em The non-convex
  {Burer-Monteiro} approach works on smooth semidefinite programs}, in Advances
  in Neural Information Processing Systems, 2016, pp.~2757--2765.

\bibitem{boumal2018deterministic}
{\sc N.~Boumal, V.~Voroninski, and A.~S. Bandeira}, {\em Deterministic
  guarantees for {Burer-Monteiro} factorizations of smooth semidefinite
  programs}, arXiv preprint arXiv:1804.02008,  (2018).

\bibitem{burer2003nonlinear}
{\sc S.~Burer and R.~D. Monteiro}, {\em A nonlinear programming algorithm for
  solving semidefinite programs via low-rank factorization}, Mathematical
  Programming, 95 (2003), pp.~329--357.

\bibitem{burer2005local}
{\sc S.~Burer and R.~D. Monteiro}, {\em Local minima and convergence in
  low-rank semidefinite programming}, Mathematical Programming, 103 (2005),
  pp.~427--444.

\bibitem{cai2019fast}
{\sc J.-F. Cai, H.~Liu, and Y.~Wang}, {\em Fast rank-one alternating
  minimization algorithm for phase retrieval}, Journal of Scientific Computing,
  79 (2019), pp.~128--147.

\bibitem{cai2017eigenvector}
{\sc Y.~Cai, L.-H. Zhang, Z.~Bai, and R.-C. Li}, {\em On an
  eigenvector-dependent nonlinear eigenvalue problem}, 2017.

\bibitem{cambier2016robust}
{\sc L.~Cambier and P.-A. Absil}, {\em Robust low-rank matrix completion by
  {R}iemannian optimization}, SIAM Journal on Scientific Computing, 38 (2016),
  pp.~S440--S460.

\bibitem{carson2017manifold}
{\sc T.~Carson, D.~G. Mixon, and S.~Villar}, {\em Manifold optimization for
  {K}-means clustering}, in 2017 International Conference on Sampling Theory
  and Applications (SampTA), IEEE, 2017, pp.~73--77.

\bibitem{chen2018proximal}
{\sc S.~Chen, S.~Ma, A.~M.-C. So, and T.~Zhang}, {\em Proximal gradient method
  for manifold optimization}, arXiv preprint arXiv:1811.00980,  (2018).

\bibitem{cho2017riemannian}
{\sc M.~Cho and J.~Lee}, {\em Riemannian approach to batch normalization}, in
  Advances in Neural Information Processing Systems, 2017, pp.~5225--5235.

\bibitem{dai2017conjugate}
{\sc X.~Dai, Z.~Liu, L.~Zhang, and A.~Zhou}, {\em A conjugate gradient method
  for electronic structure calculations}, SIAM Journal on Scientific Computing,
  39 (2017), pp.~A2702--A2740.

\bibitem{de2016new}
{\sc G.~de~Carvalho~Bento, J.~X. da~Cruz~Neto, and P.~R. Oliveira}, {\em A new
  approach to the proximal point method: convergence on general {R}iemannian
  manifolds}, Journal of Optimization Theory and Applications, 168 (2016),
  pp.~743--755.

\bibitem{dirr2007nonsmooth}
{\sc G.~Dirr, U.~Helmke, and C.~Lageman}, {\em Nonsmooth {R}iemannian
  optimization with applications to sphere packing and grasping}, in Lagrangian
  and Hamiltonian Methods for Nonlinear Control 2006, Springer, 2007,
  pp.~29--45.

\bibitem{duchi2011adaptive}
{\sc J.~Duchi, E.~Hazan, and Y.~Singer}, {\em Adaptive subgradient methods for
  online learning and stochastic optimization}, Journal of Machine Learning
  Research, 12 (2011), pp.~2121--2159.

\bibitem{EdelmanAriasSmith1999}
{\sc A.~Edelman, T.~A. Arias, and S.~T. Smith}, {\em The geometry of algorithms
  with orthogonality constraints}, SIAM J. Matrix Anal. Appl., 20 (1999),
  pp.~303--353.

\bibitem{gabay1982minimizing}
{\sc D.~Gabay}, {\em Minimizing a differentiable function over a differential
  manifold}, J. Optim. Theory Appl., 37 (1982), pp.~177--219.

\bibitem{gao2018new}
{\sc B.~Gao, X.~Liu, X.~Chen, and Y.~Yuan}, {\em A new first-order algorithmic
  framework for optimization problems with orthogonality constraints}, SIAM
  Journal on Optimization, 28 (2018), pp.~302--332.

\bibitem{GaoSun2010}
{\sc Y.~Gao and D.~Sun}, {\em A majorized penalty approach for calibrating rank
  constrained correlation matrix problems}, tech. report, National University
  of Singapore, 2010.

\bibitem{grohs2015nonsmooth}
{\sc P.~Grohs and S.~Hosseini}, {\em Nonsmooth trust region algorithms for
  locally {L}ipschitz functions on {R}iemannian manifolds}, IMA Journal of
  Numerical Analysis, 36 (2015), pp.~1167--1192.

\bibitem{hosseini2015convergence}
{\sc S.~Hosseini}, {\em Convergence of nonsmooth descent methods via
  {K}urdyka--{L}ojasiewicz inequality on {R}iemannian manifolds}, Hausdorff
  Center for Mathematics and Institute for Numerical Simulation, University of
  Bonn (2015,(INS Preprint No. 1523)),  (2015).

\bibitem{hosseini2017riemannian}
{\sc S.~Hosseini and A.~Uschmajew}, {\em A {R}iemannian gradient sampling
  algorithm for nonsmooth optimization on manifolds}, SIAM Journal on
  Optimization, 27 (2017), pp.~173--189.

\bibitem{hu2018structured}
{\sc J.~Hu, B.~Jiang, L.~Lin, Z.~Wen, and Y.~Yuan}, {\em Structured
  quasi-{N}ewton methods for optimization with orthogonality constraints},
  arXiv preprint arXiv:1809.00452,  (2018).

\bibitem{hu2016note}
{\sc J.~Hu, B.~Jiang, X.~Liu, and Z.~Wen}, {\em A note on semidefinite
  programming relaxations for polynomial optimization over a single sphere},
  Science China Mathematics, 59 (2016), pp.~1543--1560.

\bibitem{hu2017adaptive}
{\sc J.~Hu, A.~Milzarek, Z.~Wen, and Y.~Yuan}, {\em Adaptive regularized newton
  method for {R}iemannian optimization}, arXiv preprint arXiv:1708.02016,
  (2017).

\bibitem{hu2018adaptive}
{\sc J.~Hu, A.~Milzarek, Z.~Wen, and Y.~Yuan}, {\em Adaptive quadratically
  regularized {N}ewton method for {R}iemannian optimization}, SIAM J. Matrix
  Anal. Appl., 39 (2018), pp.~1181--1207.

\bibitem{huang2013optimization}
{\sc W.~Huang}, {\em Optimization algorithms on {R}iemannian manifolds with
  applications}, PhD thesis, The Florida State University, 2013.

\bibitem{huang2018riemannian}
{\sc W.~Huang, P.-A. Absil, and K.~Gallivan}, {\em A {R}iemannian {BFGS} method
  without differentiated retraction for nonconvex optimization problems}, SIAM
  J. Optim., 28 (2018), pp.~470--495.

\bibitem{huang2015riemannian}
{\sc W.~Huang, P.-A. Absil, and K.~A. Gallivan}, {\em A {R}iemannian symmetric
  rank-one trust-region method}, Math. Program., 150 (2015), pp.~179--216.

\bibitem{huang2017intrinsic}
{\sc W.~Huang, P.-A. Absil, and K.~A. Gallivan}, {\em Intrinsic representation
  of tangent vectors and vector transports on matrix manifolds}, Numerische
  Mathematik, 136 (2017), pp.~523--543.

\bibitem{huang2015broyden}
{\sc W.~Huang, K.~A. Gallivan, and P.-A. Absil}, {\em A {B}royden class of
  quasi-{N}ewton methods for {R}iemannian optimization}, SIAM J. Optim., 25
  (2015), pp.~1660--1685.

\bibitem{IanPor17}
{\sc B.~Iannazzo and M.~Porcelli}, {\em The {R}iemannian {B}arzilai-{B}orwein
  method with nonmonotone line search and the matrix geometric mean
  computation}, IMA Journal of Numerical Analysis, 00 (2017), pp.~1--23.

\bibitem{jiang2015framework}
{\sc B.~Jiang and Y.-H. Dai}, {\em A framework of constraint preserving update
  schemes for optimization on {S}tiefel manifold}, Mathematical Programming,
  153 (2015), pp.~535--575.

\bibitem{jiang2016l_p}
{\sc B.~Jiang, Y.-F. Liu, and Z.~Wen}, {\em L\_p-norm regularization algorithms
  for optimization over permutation matrices}, SIAM Journal on Optimization, 26
  (2016), pp.~2284--2313.

\bibitem{jiang2017vector}
{\sc B.~Jiang, S.~Ma, A.~M.-C. So, and S.~Zhang}, {\em Vector transport-free
  svrg with general retraction for {R}iemannian optimization: Complexity
  analysis and practical implementation}, arXiv preprint arXiv:1705.09059,
  (2017).

\bibitem{jolliffe2003modified}
{\sc I.~T. Jolliffe, N.~T. Trendafilov, and M.~Uddin}, {\em A modified
  principal component technique based on the lasso}, Journal of computational
  and Graphical Statistics, 12 (2003), pp.~531--547.

\bibitem{kass1990nonlinear}
{\sc R.~E. Kass}, {\em Nonlinear regression analysis and its applications}, J.
  Am. Stat. Assoc., 85 (1990), pp.~594--596.

\bibitem{kovnatsky2016madmm}
{\sc A.~Kovnatsky, K.~Glashoff, and M.~M. Bronstein}, {\em Madmm: a generic
  algorithm for non-smooth optimization on manifolds}, in European Conference
  on Computer Vision, Springer, 2016, pp.~680--696.

\bibitem{kressner2014low}
{\sc D.~Kressner, M.~Steinlechner, and B.~Vandereycken}, {\em Low-rank tensor
  completion by {R}iemannian optimization}, BIT Numer. Math., 54 (2014),
  pp.~447--468.

\bibitem{lai2016localized}
{\sc R.~Lai and J.~Lu}, {\em Localized density matrix minimization and
  linear-scaling algorithms}, Journal of Computational Physics, 315 (2016),
  pp.~194--210.

\bibitem{lai2014splitting}
{\sc R.~Lai and S.~Osher}, {\em A splitting method for orthogonality
  constrained problems}, Journal of Scientific Computing, 58 (2014),
  pp.~431--449.

\bibitem{lai2014folding}
{\sc R.~Lai, Z.~Wen, W.~Yin, X.~Gu, and L.~M. Lui}, {\em Folding-free global
  conformal mapping for genus-0 surfaces by harmonic energy minimization},
  Journal of Scientific Computing, 58 (2014), pp.~705--725.

\bibitem{lecun2015deep}
{\sc Y.~LeCun, Y.~Bengio, and G.~Hinton}, {\em Deep learning}, nature, 521
  (2015), p.~436.

\bibitem{li2018near}
{\sc C.~J. Li, M.~Wang, H.~Liu, and T.~Zhang}, {\em Near-optimal stochastic
  approximation for online principal component estimation}, Mathematical
  Programming, 167 (2018), pp.~75--97.

\bibitem{liu2019simple}
{\sc C.~Liu and N.~Boumal}, {\em Simple algorithms for optimization on
  {R}iemannian manifolds with constraints}, arXiv preprint arXiv:1901.10000,
  (2019).

\bibitem{liu2017subspace}
{\sc H.~Liu, J.-F. Cai, and Y.~Wang}, {\em Subspace clustering by (k, k)-sparse
  matrix factorization}, Inverse Problems \& Imaging, 11 (2017), pp.~539--551.

\bibitem{liu2014convergence}
{\sc X.~Liu, X.~Wang, Z.~Wen, and Y.~Yuan}, {\em On the convergence of the
  self-consistent field iteration in {Kohn--Sham} density functional theory},
  SIAM Journal on Matrix Analysis and Applications, 35 (2014), pp.~546--558.

\bibitem{liu2015analysis}
{\sc X.~Liu, Z.~Wen, X.~Wang, M.~Ulbrich, and Y.~Yuan}, {\em On the analysis of
  the discretized {Kohn--Sham} density functional theory}, SIAM Journal on
  Numerical Analysis, 53 (2015), pp.~1758--1785.

\bibitem{liu2013limited}
{\sc X.~Liu, Z.~Wen, and Y.~Zhang}, {\em Limited memory block {K}rylov subspace
  optimization for computing dominant singular value decompositions}, SIAM J.
  Sci. Comput., 35 (2013), pp.~A1641--A1668.

\bibitem{liu2015efficient}
{\sc X.~Liu, Z.~Wen, and Y.~Zhang}, {\em An efficient {Gauss--Newton} algorithm
  for symmetric low-rank product matrix approximations}, SIAM Journal on
  Optimization, 25 (2015), pp.~1571--1608.

\bibitem{liu2017accelerated}
{\sc Y.~Liu, F.~Shang, J.~Cheng, H.~Cheng, and L.~Jiao}, {\em Accelerated
  first-order methods for geodesically convex optimization on {R}iemannian
  manifolds}, in Advances in Neural Information Processing Systems, 2017,
  pp.~4868--4877.

\bibitem{mei2017solving}
{\sc S.~Mei, T.~Misiakiewicz, A.~Montanari, and R.~I. Oliveira}, {\em Solving
  {SDPs} for synchronization and maxcut problems via the {G}rothendieck
  inequality}, arXiv preprint arXiv:1703.08729,  (2017).

\bibitem{montanari2016non}
{\sc A.~Montanari and E.~Richard}, {\em Non-negative principal component
  analysis: Message passing algorithms and sharp asymptotics}, IEEE
  Transactions on Information Theory, 62 (2016), pp.~1458--1484.

\bibitem{NocedalWright06}
{\sc J.~Nocedal and S.~J. Wright}, {\em Numerical Optimization}, Springer
  Series in Operations Research and Financial Engineering, Springer, New York,
  second~ed., 2006.

\bibitem{oja1985stochastic}
{\sc E.~Oja and J.~Karhunen}, {\em On stochastic approximation of the
  eigenvectors and eigenvalues of the expectation of a random matrix}, Journal
  of mathematical analysis and applications, 106 (1985), pp.~69--84.

\bibitem{pataki1998rank}
{\sc G.~Pataki}, {\em On the rank of extreme matrices in semidefinite programs
  and the multiplicity of optimal eigenvalues}, Mathematics of operations
  research, 23 (1998), pp.~339--358.

\bibitem{pulay1980convergence}
{\sc P.~Pulay}, {\em Convergence acceleration of iterative sequences. the case
  of {SCF} iteration}, Chemical Physics Letters, 73 (1980), pp.~393--398.

\bibitem{pulay1982improved}
{\sc P.~Pulay}, {\em Improved {SCF} convergence acceleration}, Journal of
  Computational Chemistry, 3 (1982), pp.~556--560.

\bibitem{qi2011numerical}
{\sc C.~Qi}, {\em Numerical optimization methods on {R}iemannian manifolds},
  PhD thesis, Florida State University, 2011.

\bibitem{Ring2012Optimization}
{\sc W.~Ring and B.~Wirth}, {\em Optimization methods on {R}iemannian manifolds
  and their application to shape space}, SIAM J. Optim., 22 (2012),
  pp.~596--627.

\bibitem{sato2017riemannian}
{\sc H.~Sato, H.~Kasai, and B.~Mishra}, {\em Riemannian stochastic variance
  reduced gradient}, arXiv preprint arXiv:1702.05594,  (2017).

\bibitem{schoen1997lectures}
{\sc R.~M. Schoen and S.-T. Yau}, {\em Lectures on harmonic maps}, vol.~2, Amer
  Mathematical Society, 1997.

\bibitem{seibert2013properties}
{\sc M.~Seibert, M.~Kleinsteuber, and K.~H{\"u}per}, {\em Properties of the
  {BFGS} method on {R}iemannian manifolds}, Mathematical System Theory C
  Festschrift in Honor of Uwe Helmke on the Occasion of his Sixtieth Birthday,
  (2013), pp.~395--412.

\bibitem{shamir2015stochastic}
{\sc O.~Shamir}, {\em A stochastic {PCA} and {SVD} algorithm with an
  exponential convergence rate}, in International Conference on Machine
  Learning, 2015, pp.~144--152.

\bibitem{SimonAbell2010}
{\sc D.~Simon and J.~Abell}, {\em A majorization algorithm for constrained
  correlation matrix approximation}, Linear Algebra Appl., 432 (2010),
  pp.~1152--1164.

\bibitem{singer2011three}
{\sc A.~Singer and Y.~Shkolnisky}, {\em Three-dimensional structure
  determination from common lines in cryo-em by eigenvectors and semidefinite
  programming}, SIAM journal on imaging sciences, 4 (2011), pp.~543--572.

\bibitem{smith1994optimization}
{\sc S.~T. Smith}, {\em Optimization techniques on {R}iemannian manifolds},
  Fields Institute Communications, 3 (1994).

\bibitem{sun2006optimization}
{\sc W.~Sun and Y.~Yuan}, {\em Optimization theory and methods: nonlinear
  programming}, vol.~1, Springer Science \& Business Media, 2006.

\bibitem{toth2017local}
{\sc A.~Toth, J.~A. Ellis, T.~Evans, S.~Hamilton, C.~Kelley, R.~Pawlowski, and
  S.~Slattery}, {\em Local improvement results for {A}nderson acceleration with
  inaccurate function evaluations}, SIAM Journal on Scientific Computing, 39
  (2017), pp.~S47--S65.

\bibitem{udriste1994convex}
{\sc C.~Udriste}, {\em Convex functions and optimization methods on
  {R}iemannian manifolds}, vol.~297, Springer Science \& Business Media, 1994.

\bibitem{ulbrich2015proximal}
{\sc M.~Ulbrich, Z.~Wen, C.~Yang, D.~Klockner, and Z.~Lu}, {\em A proximal
  gradient method for ensemble density functional theory}, SIAM Journal on
  Scientific Computing, 37 (2015), pp.~A1975--A2002.

\bibitem{vandereycken2013low}
{\sc B.~Vandereycken}, {\em Low-rank matrix completion by {R}iemannian
  optimization}, SIAM Journal on Optimization, 23 (2013), pp.~1214--1236.

\bibitem{vishnoi2018geodesic}
{\sc N.~K. Vishnoi}, {\em Geodesic convex optimization: Differentiation on
  manifolds, geodesics, and convexity}, arXiv preprint arXiv:1806.06373,
  (2018).

\bibitem{waldspurger2015phase}
{\sc I.~Waldspurger, A.~d’Aspremont, and S.~Mallat}, {\em Phase recovery,
  maxcut and complex semidefinite programming}, Mathematical Programming, 149
  (2015), pp.~47--81.

\bibitem{wang2019global}
{\sc Y.~Wang, W.~Yin, and J.~Zeng}, {\em Global convergence of admm in
  nonconvex nonsmooth optimization}, Journal of Scientific Computing, 78
  (2019), pp.~29--63.

\bibitem{wei2016guarantees}
{\sc K.~Wei, J.-F. Cai, T.~F. Chan, and S.~Leung}, {\em Guarantees of
  {R}iemannian optimization for low rank matrix recovery}, SIAM Journal on
  Matrix Analysis and Applications, 37 (2016), pp.~1198--1222.

\bibitem{wen2013adaptive}
{\sc Z.~Wen, A.~Milzarek, M.~Ulbrich, and H.~Zhang}, {\em Adaptive regularized
  self-consistent field iteration with exact {H}essian for electronic structure
  calculation}, SIAM J. Sci. Comput., 35 (2013), pp.~A1299--A1324.

\bibitem{wen2016trace}
{\sc Z.~Wen, C.~Yang, X.~Liu, and Y.~Zhang}, {\em Trace-penalty minimization
  for large-scale eigenspace computation}, Journal of Scientific Computing, 66
  (2016), pp.~1175--1203.

\bibitem{wen2013feasible}
{\sc Z.~Wen and W.~Yin}, {\em A feasible method for optimization with
  orthogonality constraints}, Math. Program., 142 (2013), pp.~397--434.

\bibitem{wen2012solving}
{\sc Z.~Wen, W.~Yin, and Y.~Zhang}, {\em Solving a low-rank factorization model
  for matrix completion by a nonlinear successive over-relaxation algorithm},
  Mathematical Programming Computation, 4 (2012), pp.~333--361.

\bibitem{wen2017accelerating}
{\sc Z.~Wen and Y.~Zhang}, {\em Accelerating convergence by augmented
  {Rayleigh--Ritz} projections for large-scale eigenpair computation}, SIAM
  Journal on Matrix Analysis and Applications, 38 (2017), pp.~273--296.

\bibitem{wu2015regularized}
{\sc X.~Wu, Z.~Wen, and W.~Bao}, {\em A regularized newton method for computing
  ground states of {Bose-Einstein} condensates}, arXiv preprint
  arXiv:1504.02891,  (2015).

\bibitem{xiao2018regularized}
{\sc X.~Xiao, Y.~Li, Z.~Wen, and L.~Zhang}, {\em A regularized semi-smooth
  newton method with projection steps for composite convex programs}, Journal
  of Scientific Computing,  (2018), pp.~1--26.

\bibitem{xie2018non}
{\sc T.~Xie and F.~Chen}, {\em Non-convex clustering via proximal alternating
  linearized minimization method}, International Journal of Wavelets,
  Multiresolution and Information Processing, 16 (2018), p.~1840013.

\bibitem{yang2014optimality}
{\sc W.~H. Yang, L.-H. Zhang, and R.~Song}, {\em Optimality conditions for the
  nonlinear programming problems on {R}iemannian manifolds}, Pacific Journal of
  Optimization, 10 (2014), pp.~415--434.

\bibitem{yuan2017global}
{\sc H.~Yuan, X.~Gu, R.~Lai, and Z.~Wen}, {\em Global optimization with
  orthogonality constraints via stochastic diffusion on manifold}, arXiv
  preprint arXiv:1707.02126,  (2017).

\bibitem{zass2007nonnegative}
{\sc R.~Zass and A.~Shashua}, {\em Nonnegative sparse pca}, in Advances in
  neural information processing systems, 2007, pp.~1561--1568.

\bibitem{ZhaHag04}
{\sc H.~Zhang and W.~W. Hager}, {\em A nonmonotone line search technique and
  its application to unconstrained optimization}, SIAM J. Optim., 14 (2004),
  pp.~1043--1056.

\bibitem{zhang2016riemannian}
{\sc H.~Zhang, S.~J. Reddi, and S.~Sra}, {\em {Riemannian SVRG}: Fast
  stochastic optimization on {R}iemannian manifolds}, in Advances in Neural
  Information Processing Systems, 2016, pp.~4592--4600.

\bibitem{zhang2016first}
{\sc H.~Zhang and S.~Sra}, {\em First-order methods for geodesically convex
  optimization}, in Conference on Learning Theory, 2016, pp.~1617--1638.

\bibitem{zhang2017sparse}
{\sc J.~Zhang, H.~Liu, Z.~Wen, and S.~Zhang}, {\em A sparse completely positive
  relaxation of the modularity maximization for community detection}, SIAM J.
  Sci. Comput., 40 (2018), pp.~A3091--A3120.

\bibitem{zhang2017primal}
{\sc J.~Zhang, S.~Ma, and S.~Zhang}, {\em Primal-dual optimization algorithms
  over {R}iemannian manifolds: an iteration complexity analysis}, arXiv
  preprint arXiv:1710.02236,  (2017).

\bibitem{zhang2017subspace}
{\sc J.~Zhang, Z.~Wen, and Y.~Zhang}, {\em Subspace methods with local
  refinements for eigenvalue computation using low-rank tensor-train format},
  Journal of Scientific Computing, 70 (2017), pp.~478--499.

\bibitem{zhang2018cubic}
{\sc J.~Zhang and S.~Zhang}, {\em A cubic regularized {N}ewton's method over
  {R}iemannian manifolds}, arXiv preprint arXiv:1805.05565,  (2018).

\bibitem{zhang2015maximization}
{\sc L.~Zhang and R.~Li}, {\em Maximization of the sum of the trace ratio on
  the {S}tiefel manifold, ii: Computation}, Science China Mathematics, 58
  (2015), pp.~1549--1566.

\bibitem{zhang2014gradient}
{\sc X.~Zhang, J.~Zhu, Z.~Wen, and A.~Zhou}, {\em Gradient type optimization
  methods for electronic structure calculations}, SIAM Journal on Scientific
  Computing, 36 (2014), pp.~C265--C289.

\bibitem{zhang2017global}
{\sc Y.~Zhang, Y.~Lau, H.-w. Kuo, S.~Cheung, A.~Pasupathy, and J.~Wright}, {\em
  On the global geometry of sphere-constrained sparse blind deconvolution}, in
  Proceedings of the IEEE Conference on Computer Vision and Pattern
  Recognition, 2017, pp.~4894--4902.

\bibitem{zhao2015riemannian}
{\sc Z.~Zhao, Z.-J. Bai, and X.-Q. Jin}, {\em A {R}iemannian newton algorithm
  for nonlinear eigenvalue problems}, SIAM Journal on Matrix Analysis and
  Applications, 36 (2015), pp.~752--774.

\bibitem{zhu2017riemannian}
{\sc X.~Zhu}, {\em A {R}iemannian conjugate gradient method for optimization on
  the {S}tiefel manifold}, Computational Optimization and Applications, 67
  (2017), pp.~73--110.

\end{thebibliography}
\end{document}